\documentclass{article}
\usepackage{standalone}
\usepackage[utf8]{inputenc} 
\usepackage{fixltx2e} 

\usepackage{tikz}
\usepackage{graphicx}
\graphicspath{{Figures/}}

\usepackage{parskip} 
\usepackage{microtype} 

\usepackage{amsmath}
\usepackage{amsthm}
\usepackage{amsbsy}
\usepackage{amsopn}
\usepackage{amssymb}
\usepackage{amsfonts}
\usepackage[official]{eurosym}

\makeatletter
\def\thm@space@setup{%
  \thm@preskip=\parskip \thm@postskip=0pt
}
\makeatother

\usepackage{color}
\definecolor{lightgray}{gray}{0.9}
\definecolor{red}{rgb}{1,0,0}

\usepackage[backend=bibtex,hyperref=true,backref=true,maxcitenames=3,url=true,style=alphabetic]{biblatex}
\usepackage[pdfencoding=unicode]{hyperref}
\theoremstyle{plain}
\newtheorem{theorem}{Theorem}[section]
\newtheorem{lemma}[theorem]{Lemma}
\newtheorem{proposition}[theorem]{Proposition}

\newtheorem{claim}[theorem]{Claim}
\theoremstyle{definition}
\newtheorem{definition}[theorem]{Definition}

\newtheorem{remark}[theorem]{Remark}

\newcommand{\eps}{\varepsilon}
\newcommand{\cA}{\mathcal{A}}

\newcommand{\cL}{\mathcal{L}}

\newcommand{\cP}{\mathcal{P}}

\newcommand{\cU}{\mathcal{U}}

\newcommand{\bC}{\mathbb{C}}
\newcommand{\bF}{\mathbb{F}}
\newcommand{\bH}{\mathbb{H}}

\newcommand{\bN}{\mathbb{N}}

\newcommand{\bR}{\mathbb{R}}

\newcommand{\bZ}{\mathbb{Z}}

\newcommand{\rH}{\mathrm{H}}

\newcommand{\ph}{\varphi}

\newcommand{\id}{\mathbbm 1}

\DeclareMathOperator{\Ham}{Ham}

\DeclareMathOperator{\sign}{sign}
\newcommand{\del}{\partial}
\newcommand{\oB}{\overline{B}}

\newcommand{\R}{\mathbb{R}}
\newcommand{\Z}{\mathbb{Z}}
\DeclareMathOperator{\trace}{trace}
\DeclareMathOperator{\spec}{spec}
\newcommand{\tf}{{\widetilde{f}}}

\newcommand{\talpha}{\widetilde{\alpha}}

\def\id{{1\hskip-2.5pt{\rm l}}}

\title{Embeddings of free groups into asymptotic cones of Hamiltonian diffeomorphisms}
\date{\today}
\author{D. Alvarez-Gavela, V. Kaminker, A. Kislev, K. Kliakhandler,  \\ A. Pavlichenko, L. Rigolli, D. Rosen, O. Shabtai, \\  B. Stevenson, J. Zhang \footnote{This paper was the outcome of the authors' work in the computational symplectic topology graduate student team-based research program held May 17 to May 26, 2015 in Tel-Aviv University and July 27 to August, 5, 2015 at ICERM. The program was partially supported by the European Research Council (ERC) grants No. 637386 and No. 338809.} }

\begin{document}
\maketitle
\begin{abstract}
  Given a symplectic surface $(\Sigma, \omega)$ of genus $g \geq 4$, we show that the free group with two generators embeds into every asymptotic cone of $(\text{Ham}(\Sigma, \omega),d_{\text{H} })$, where $d_H$ is the Hofer metric. The result stabilizes to products with symplectically aspherical manifolds.
\end{abstract}
\tableofcontents

\section{Introduction and statement of results}
\label{introduction}
This paper combines ideas from geometric group theory and symplectic topology.
Our goal is to study the coarse geometry of the group of Hamiltonian diffeomorphisms of a symplectic manifold, equipped with the Hofer norm.
When the manifold is a symplectic surface of genus $g \geq 4$, we show, using Hamiltonian Floer theory, that any asymptotic cone of this group contains a free subgroup with 2 generators. The same holds for a product of such a surface with a symplectically aspherical manifold.

\subsection{The asymptotic cone}
Given a metric space $(X,d)$ with a chosen base point $x_0 \in X$, its asymptotic cone is a rescaling limit which informally results from looking at $X$ from points further and further away from the basepoint.
More precisely, given a non-principal ultrafilter $\cU$ on $\mathbb{N}$, we construct the asymptotic cone $\text{Cone}_{\cU}(X,d)$ as follows.
The elements of $\text{Cone}_{\cU}(X,d)$ consist of equivalence classes of sequences $(x_j)_{j \in \bN}$ of points $x_j \in X$ such that the sequence of real numbers $d(x_0,x_j)/j$ is bounded, where we identify $(x_j)_{j \in \bN} \sim (y_j)_{j \in \bN}$ whenever $\lim_{\cU} d(x_j,y_j)/j = 0$
The asymptotic cone is a metric space, with the metric given by the ultralimit $d_{\cU}( [(x_j)_{j \in \bN} ] , [ (y_j)_{j \in \bN} ] ) = \lim_{\cU} d(x_j,y_j)/j$.
For example, $\text{Cone}_{\cU}(\bR^n)=\text{Cone}_{\cU}(\bZ^n)=\bR^n$ for every $\cU$, whereas $\text{Cone}_{\cU}(\bH)$ is a tree where every vertex has an uncountable degree.

Let $G$ be a group equipped with a bi-invariant metric $d$, so that $d(gh,g'h)=d(g,g')=d(hg,hg')$ for all $g,g',h \in G$.
It is easy to check that the asymptotic cone $\text{Cone}_{\cU}(G,d)$ is also a group, where the group operation is multiplication in each coordinate (see \cite[Proposition 3.3]{Calegari2011}).
Note that for this operation to be well defined it is crucial for the metric to be bi-invariant.
The asymptotic cone construction was introduced by Gromov and is nowadays a standard object in geometric group theory, see for example \cite{Gromov1993} and \cite{Calegari2011}.

In this paper we consider the asymptotic cone in the context of symplectic geometry.

\subsection{Hofer's geometry}
Let $(M, \omega)$ be a symplectic manifold.
Given a smooth function $H:S^1 \times M \rightarrow \bR$, there exists a unique (time-dependent) vector field $X_t$ on $M$ such that $\omega(X_{H_t}, \cdot ) = -dH_t$, where $H_t(p):=H(t,p)$.
Let $\varphi^t_H:M \to M$ be the flow of the ODE $\dot x(t)=X_{t}(x(t))$, making sufficient assumptions to ensure that the flow is globally defined on the time interval $[0,1]$ (for example, we could take $M$ to be compact).
Inside the group $\text{Symp}(M, \omega)=\{ \phi \in \text{Diff}(M) : \, \, \phi^* \omega = \omega\}$ of symplectomorphisms we have the subgroup of Hamiltonian diffeomorphisms $\Ham(M, \omega)$, which consists of the time-one maps  $\varphi^1_H:M \to M$ of flows as above.
The group $\text{Ham}(M, \omega)$ is equipped with a geometrically meaningful bi-invariant metric introduced by Hofer.
The resulting metric group is an important object of study in symplectic geometry. For $\phi \in \text{Ham}(M,\omega)$, we define the Hofer norm
\[ \| \phi \|_{\text{H}} = \inf_{H} \int_{0}^{1} \left(\max\limits_{M}H_t - \min\limits_{M} H_t\right) \, \,  dt \]
where the infimum ranges over all functions $H \in C^{\infty}(S^1\times M)$ such that $\varphi^1_H=\phi$. The Hofer norm was introduced by Hofer in \cite{Hofer1990}, where it was proven to be non-degenerate for $M=\R^{2n}$.
Later, the non-degeneracy was generalized to all symplectic manifolds (see \cite{Viterbo1992}, \cite{Polterovich1993}, \cite{Lalonde1995}).
From this norm we obtain a non-degenerate metric $d_{\text{H}}(\phi, \psi)=\| \phi \circ \psi^{-1} \|_{\text{H} }$, called the Hofer metric.
The bi-invariance of the Hofer metric follows from the conjugation invariance of the Hofer norm.

\subsection{Main results}
The main result of this paper is the following theorem. Let $\mathbb{F}_2=\mathbb{F}\langle V,H \rangle$ denote the (non abelian) free group on two letters.

\begin{theorem}
\label{cone}
Let $(\Sigma, \omega_\Sigma)$ be a closed symplectic surface of genus $g \geq 4$ and let $(M,\omega_M)$ be a closed symplectically aspherical manifold. Then, for any non-principal ultrafilter $\cU$ on $\bN$, there exists a monomorphism $\bF_2 \hookrightarrow \text{Cone}_{\cU}(\text{Ham}(\Sigma \times M), d_{\text{H} })$.
\end{theorem}

In the literature there are results related to abelian embeddings into $\text{Ham}$. See for example \cite{Py2008} and \cite{Usher2011}.
As far as we know, our theorem is the first one considering non-abelian embeddings.
To prove our result, we construct a sequence of homomorphisms $\Phi_{k}:\bF_2 \to \text{Ham}(\Sigma \times M)$ such that for any nontrivial word $w \in \bF_2$ we have $ \| \Phi_{k}(w) \|_{\text{H} }  \to \infty$ as $k \to \infty$, with asymptotic growth linear in $k$.
This allows us to construct the embedding
\[ \Phi :  \mathbb{F}_2 \hookrightarrow \text{Cone}_{\cU}( \text{Ham}(\Sigma \times M) , d_{\text{H} }) , \quad
  w \mapsto [ (\Phi_k(w) )_{k \in \bN} ] . \]
This map is clearly a group homomorphism.
Injectivity follows from the above asymptotics. Indeed, for any nontrivial word $w \in \mathbb{F}_2$ the linear growth of the Hofer norm implies that
\[\lim_{\cU} \frac{\| \Phi_k(w) \|_{\text{H} }}{k} \geq  \liminf_{k \to \infty} \frac{\| \Phi_k(w) \|_{\text{H} } }{ k}  > 0.\]

Below, we will make extensive use of the following definition:
\begin{definition}
\label{long}
We will say that a word $w \in \bF \langle V,H \rangle$ is \begin{it}long\end{it} if it is not conjugate to a power of $V$ or $H$.
\end{definition}

\subsection{The induced norm on \texorpdfstring{$\bF \left\langle V,H \right\rangle$}{F<V,H>}}
We now discuss the geometry of the embedding $\Phi : \bF_2 \hookrightarrow \text{Cone}_{\cU}(\text{Ham}(\Sigma \times M), d_{\text{H} })$.
Consider the induced conjugation invariant norm $c_{\cU,\Phi}$ on $\bF_2$ obtained by pulling back the conjugation invariant norm on $\text{Cone}_{\cU}(\text{Ham}(\Sigma \times M), d_{\text{H} })$ by the homomorphism $\Phi$.
Theorem \ref{cone} exhibits a major difference between the Hofer norm and the commutator length (and its induced norm) on the group of Hamiltonian diffeomorphisms, and boils down to the fact that $c_{\cU, \Phi}$ is indeed a norm, i.e $c_{\cU,\Phi}(w) > 0$ for any nontrivial word $w \in \bF_2$.

Banyaga proved \cite{banyaga1997} that $\text{Ham}(\Sigma  \times M)$ is a simple group - hence its commutator subgroup equals the entire group.
We can therefore define for any Hamiltonian diffeomorphism $\phi$ its commutator length $CL(\phi)$ as the smallest non negative integer $k$ such that one can write
\[ \phi = \prod_{j=1}^k \psi_j \eta_j \psi_j^{-1} \eta_j^{-1}, \qquad \psi_j, \eta_j \in \text{Ham}(\Sigma  \times M).\]
For a non-principal ultrafilter $\cU$, consider $\text{Cone}_{\cU}( \text{Ham}( \Sigma  \times M), d_{CL} )$, the asymptotic cone with respect to the metric $d_{CL}(\phi, \psi)=CL(\phi \circ \psi^{-1})$.
Observe that the metric $d_{CL}$ is bi-invariant.
\begin{claim}
The group $\text{Cone}_{\cU}( \text{Ham}( \Sigma  \times M), d_{CL} )$ is abelian.
\end{claim}
Indeed, take any two elements $[( \phi_j )_{j\in \bN} ], [ ( \psi_j )_{j\in \bN} ] \in \text{Cone}_{\cU}( \text{Ham}( \Sigma  \times M), d_{CL} )$.
Observe that
$d_{CL}( \phi_j \psi_j , \psi_j \phi_j) = CL( \phi_j \psi_j \phi_j^{-1} \psi_j^{-1} ) \leq 1$ for all $j \in \bN$.
It follows that
\[ \lim_{j \to \infty} \frac{ d_{CL}( \phi_j \psi_j, \psi_j \phi_j )}{j} =0, \]
and hence $[(\phi_j )_{j\in \bN} ] [ ( \psi_j )_{j\in \bN} ] = [ ( \phi_j \psi_j )_{j\in \bN} ] = [ ( \psi_j \phi_j )_{j\in \bN} ]= [ ( \psi_j )_{j\in \bN} ] [ (\phi_j )_{j\in \bN} ]$, as desired.
This is in sharp contrast with the extremely non-abelian nature of $\text{Cone}_{\cU}( \text{Ham}( \Sigma \times M), d_{\text{H}} )$, exhibited by the fact that it contains a homomorphic embedding of $\bF_2$.

For a long word $ w\in \bF \langle V,H \rangle$, define the quantity $\tau(w)$ as follows.
Choose a word of the form $H^{M_r}V^{N_r} \cdots H^{M_1}V^{N_1}$ in the conjugacy class of $w$, with exponents $N_j,M_j \neq 0$ for all $j=1, \ldots , r$.
Such a word is unique up to cyclic permutations (see Subsection \ref{subsecFindingFixedPts}).
Set
\begin{equation}\label{eq-tau}
\tau(w):=\min \{ |N_1|, \ldots , |N_r|, |M_1|, \ldots , |M_r| \}.
\end{equation}
Note that $\tau(w) > 0$ for any long word.

\begin{theorem}\label{bounds}
For the monomorphism $\Phi : \bF_2 \to \text{Cone}_{\cU}(\text{Ham}(\Sigma \times M), d_{\text{H} })$ considered above,
there exists a constant $A>0$ such that for every non-principal ultrafilter $\cU$ on $\bN$ and for every long word $w \in \bF_2$ we have
\[  A \, \tau(w) \leq c_{\cU,\Phi}(w). \]
\end{theorem}

This difference between the Hofer norm and the commutator length was already apparent in Khanevsky's work \cite{Khanevsky2014} on simple commutators with arbitrarily large Hofer norm.

\subsection{Earlier results}
To prove the theorems stated above, we study the dynamics of Hamiltonian egg-beater maps, generalizing a construction of Polterovich and Shelukhin \cite{Polterovich2015}. They work on a symplectic surface $(\Sigma, \omega)$ of genus $g \geq 4$.
One of the main results of their paper can be reformulated as follows:

For the specific word $w=HV \in \bF_2\langle V,H \rangle$ one has $\| \Phi_k(w)\|_{H} \to \infty$ as $k \to \infty$, where $\Phi_k$ is the sequence of homomorphisms under consideration. In fact, Polterovich and Shelukhin prove a stronger result using refined invariants involving persistent homology and equivariance, which survives stabilization by a closed symplectically aspherical manifold.

\subsection{Relation to other semi-norms}
Let us consider two conjugation invariant norms that bound $c_{\cU,\Phi}$ from above.
The first one is
\[ c(w)=\limsup_{k \to \infty} \frac{\| \Phi_k(w) \|_{\text{H} } }{ k}.  \]
Another natural conjugation invariant norm on $\bF_2$ is word length with respect to the smallest conjugation-invariant set of generators containing $V,H,V^{-1}$ and $H^{-1}$.
Explicitly, if $w \in \bF_2=\bF \langle V,H \rangle $, set $||| w |||$ to be the smallest non negative integer $k$ such that one can write
\[ w= \prod_{j=1}^k h_j g_j h_j^{-1} \]
where $h_j \in \bF \langle V,H \rangle $ is any word and $g_j \in \{ V,H, V^{-1}, H^{-1} \}$.
It is obvious that $c_{\cU,\Phi}(w) \leq c(w) \leq B ||| w ||| $ for some constant $B > 0$.
Indeed, one can check that any conjugation invariant norm on $\bF_2$ is bounded from above by a constant times $||| \cdot |||$. Observe also that a word $w $ is long if and only if $|||w||| > 1$.

The following observation was communicated to us by M. Khanevsky:

Consider the semi-norm $\eta$ on $\mathbb{F} \langle V, H \rangle$ which on a word $w=H^{M_r}V^{N_r} \cdots H^{M_1}V^{N_1}$ takes the value
\[ \eta(w) = | \sum_{j=1}^r N_j | + | \sum_{j=1}^r M_j |. \]
This is the pullback of the $l^1$ norm on $\bZ^2$ by the abelianization homomorphism $\bF_2 \to \bF_2 / [\bF_2, \bF_2] \simeq \bZ^2$.
One can use the energy-capacity inequality (see for example \cite{Usher2010} or \cite{Lalonde1995}) on the universal cover of the surface to prove that there exists a constant $D>0$ such that $c_{\cU,\Phi}(w) \geq D \, \eta(w)$ for all words $w \in \mathbb{F} \langle V ,H \rangle$.
For words such that $\eta(w) > 0$ the linear growth $\| \Phi_k(w) \|_{\text{H} } \to \infty$ as $k \to \infty$ follows easily from this observation.
Additionally, the inequality $c_{\cU,\Phi} \geq D \eta$ implies that the norms $c_{\cU,\Phi}$ are stably unbounded.
Note, however, that this approach for proving linear growth of the Hofer norm fails whenever $\eta(w)=0$, such as words in the commutator subgroup.

\subsection{Idea of the proof}
The construction of the homomorphism $\Phi_k : \bF \langle V,H \rangle \to \text{Ham}(\Sigma \times M)$ is a modification of the construction of Hamiltonian egg-beater maps (see \cite{Polterovich2015}, \cite{Franjione1992}). Since the proof of the main theorem follows immediately from the proof for the case of a surface $(\Sigma,\omega)$ (i.e., when $M$ is a point), let us explain the construction for this case.
We construct two Hamiltonian diffeomorphisms of $(\Sigma, \omega)$, $f_V$ and $f_H$, which are symplectic shear maps around two annuli $C_V,C_H \subset \Sigma$ intersecting each other twice. We define $\Phi_k$ by mapping $V \mapsto f_V^k$ and $H \mapsto f_H^k$.
For a long word $w$, we use a theorem of Usher \cite{Usher2011}, stating that the boundary depth of the Floer barcode is a lower bound for the Hofer norm.
By a careful study of the dynamics of the Hamiltonian diffeomorphism $\Phi_k(w)$, we can extract enough information about its filtered Floer complex to bound the boundary depth by an action gap from below and hence by Usher's theorem we get a lower bound for the Hofer norm of $\Phi_k(w)$.
Furthermore, the action gap is shown to be proportional to $k$, hence proving the linear growth of $\| \Phi_k(w) \|_{\text{H} }$ as $k \to \infty$. For words conjugate to powers of $H$ or $V$ (which is the easy case), we give a proof using Floer homological spectral invariants.

The outline of the paper is the following: In Section \ref{setup} we set the stage, construct the egg-beater maps, state precisely the linear growth of the Hofer norm and derive Theorems \ref{cone} and \ref{bounds}. In Section \ref{Floer} we give a brief recap of Floer theory. In Section \ref{fixed} we carry out the computation of the fixed points of the egg-beater maps. Finally, in Section \ref{action-cz} we calculate the necessary Floer data to obtain the asymptotic bounds on the boundary depth.

\subsection{Acknowledgements}
This work arose from one of the projects in the computational symplectic topology graduate student team-based research program hosted by Tel-Aviv University and ICERM during the summer of 2015.
The authors are very grateful to the mentors of this program, S. Borman, R. Hind, M. Khanevsky, Y. Ostrover, L. Polterovich, and M. Usher for providing a stimulating work environment and to the host institutions for their warm hospitality.
We are especially thankful to M. Usher for his lectures on persistent homology, to E. Shelukhin and M. Khanevsky for helpful advice on the paper, and to L. Polterovich for providing invaluable guidance and insight throughout the project.

\section{The Egg-beater map}\label{setup}

\subsection{Setup}
We begin by describing the local model in which we will be working. Let $L>4$ and take two copies $C_H$ and $C_V$ of the annulus $C_0=[-1,1] \times \bR/ L \bZ$ (equipped with coordinates $(x,[y])$ and symplectic form $dx \wedge dy$). The squares $S_0=[-1,1] \times ([-1,1] + L\bZ)/ L \bZ$ and $S_1=[-1,1] \times ([L/2 -1, L/2 +1]+ L\bZ) / L \bZ$ give four squares: $S_{H,0},S_{H,1} \subset C_H$ and $S_{V,0}, S_{V,1} \subset C_V$. We define $(C, \omega_0)$ to be the symplectic surface with boundary obtained by gluing together $C_H$ and $C_V$ (see figure \ref{fig-C_H-C_V}) via the symplectomorphism $VH_{0,1} :S_{V,0} \sqcup S_{V,1} \to S_{H,0} \sqcup S_{H,1}$ given by
\[ VH_{0,1}=(VH_0:S_{V,0} \to S_{H,0}) \sqcup( VH_1:S_{V,1} \to S_{H,1}),\] \[ VH_0(x,[y])=(-y,[x]) , \qquad VH_1(x,[L/2+y])=(-y,[L/2+x]).\]

\begin{figure}[h]
\centering
\includegraphics[scale=1.0]{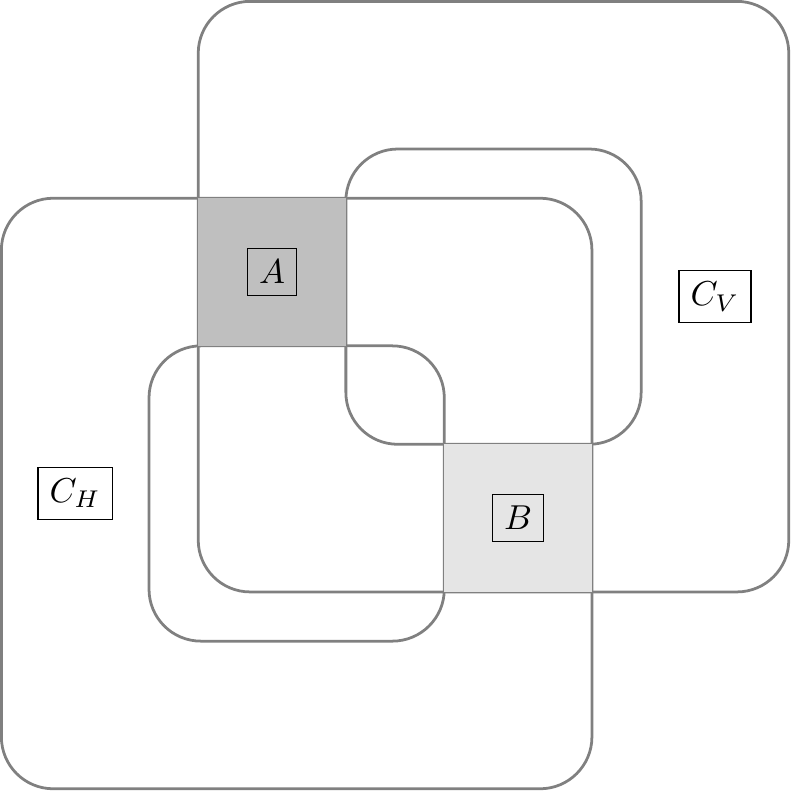}
\caption{The gluing of $C_H$ and $C_V$}
\label{fig-C_H-C_V}
\end{figure} 

We consider symplectic embeddings $(C, \omega_0) \hookrightarrow (\Sigma, \omega)$, where $(\Sigma,\omega)$ is a symplectic surface, such that the induced maps $\pi_0(\cL C) \to \pi_0(\cL \Sigma)$ and $\pi_1(C) \to \pi_1(\Sigma)$ are injective and such that each component of $\partial C$ separates $\Sigma$. Here and below we denote $\cL X=C^0(S^1,X)$, the free loop space of a space $X$. For example, embed $C$ symplectically into a large ball in $(\bR^2,dx \wedge dy)$, then embed the ball into $(S^2,\omega)$ for some area form $\omega$ and finally attach a $1-$handle to each of the four components of $S^2 \setminus  C$ to obtain a surface $\Sigma$ of genus $g=4$ with the required embedding $C \hookrightarrow \Sigma$. Up to some rescaling, this idea can be used to symplectically embed the region $C$ into any surface $(\Sigma, \omega)$ of genus $g \geq 4$.
\subsection{The shear map}\label{shear-map}
Consider the homeomorphism $f_{0}:C_0 \to C_0$ of the annulus $C_0=[-1,1] \times \bR/ L \bZ$ given by $f_{0}(x,[y])=(x,[y + 2L u_0(x)])$, where $u_0:[-1,1] \to \bR$ is the function $u_0(s)=1-|s|$ (see figure \ref{fig-u0}). 

\begin{figure}[h]
\centering
\includegraphics[scale=1.0]{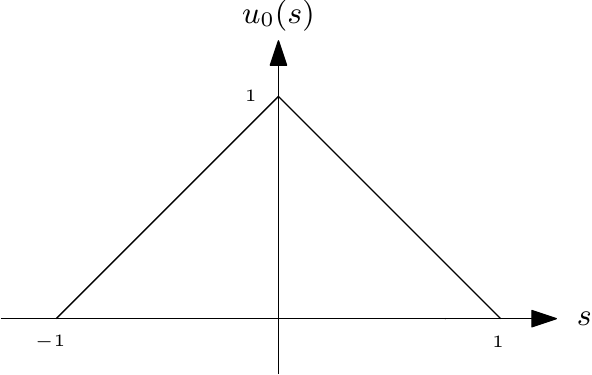}
\caption{The function $u_0$.}
\label{fig-u0}
\end{figure} 

In Sections \ref{fixed} and \ref{action-cz} below we will exploit the piecewise-linear nature of $f_{0}$ to study its dynamics explicitly. However, at the end of the day we wish to extract consequences about Hamiltonian diffeomorphisms. Therefore we will also consider $C^0-$close smoothings $f_{\delta}:C_0 \to C_0$ of $f_0$ as follows.

Replace $u_0$ by a $C^0-$close smoothing $u_{\delta} \in C^{\infty}([-1,1], \bR)$ supported in $(-1,1)$ with the following properties (see figure \ref{fig-u_delta}).
\begin{itemize}
\item $u=u_0$ outside of $U_\delta=(-1,-1+\delta) \cup(-\delta,\delta) \cup(1-\delta,1)$.
\item $0 \leq u_{\delta} \leq 1$.
\item $u_{\delta}(x) = u_{\delta}(-x)$.
\item The function $\int_{-1}^s u_{\delta}(\sigma)-u_0(\sigma) d\sigma$ is supported on $U_\delta$.
\item $|u_0(x) - u_\delta(x)| < \delta$ for all $x \in [-1,1]$.
\end{itemize}

\begin{figure}[h]
\centering
\includegraphics[scale=1.2]{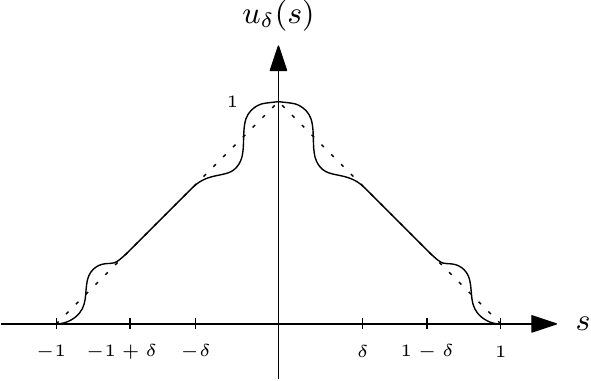}
\caption{The smoothing $u_\delta$.}
\label{fig-u_delta}
\end{figure} 

The formula $f_{\delta}(x,[y])=(x,[y+ 2L u_{\delta}(x)])$ yields a diffeomorphism $f_{\delta} :C_0 \rightarrow C_0$ which is the time-one map of the autonomous Hamiltonian flow $f_{\delta}^t$ corresponding to the function $h_{\delta}(x,[y])=-L + 2L \int_{-1}^x  u_{\delta}(\sigma)d\sigma$. Explicitly:
\[ f_{\delta}^t(x,[y])=(x,[y+ 2L t u_{\delta}(x) ] ).\]
The above conditions on $u_\delta$ will help us show that the relevant dynamics of $f^t_{\delta}$ can be extracted from the dynamics of $f^t_0$ for $\delta>0$ sufficiently small.
\subsection{The egg-beater homomorphism}\label{sec-eb-morphism}
The Hamiltonian flows $f^t_{\delta}$ determine two Hamiltonian flows $f^t_{H,\delta}$ and $f^t_{V,\delta}$ of $C$, by shearing either the annulus $C_H \subset C$ or the annulus $C_V \subset C$. We can extend these to the identity outside $\Sigma$ to obtain Hamiltonian diffeomorphisms $f^t_{V,\delta}$ and $f^t_{H,\delta}$ of $\Sigma$. Note that indeed they are time-one maps of autonomous Hamiltonian flows since each component of $\partial C$ separates $\Sigma$ and hence the functions $h_{H,\delta}$ and $h_{V, \delta} $ inducing the flows can be extended to $\Sigma$ by constants in a consistent way. Consider the sequence of homomorphisms
\[ \Phi_k: \mathbb{F}\langle V,H \rangle  \to \text{Ham}(\Sigma \times M), \quad V \mapsto f_{V,1/k}^{k} \times \id, \quad H \mapsto f_{H,1/k}^{k} \times \id.\]
The image of a general word $w=H^{M_r}V^{N_r} \cdots H^{M_1}V^{N_1}$ under the map $\Phi_k$ is the Hamiltonian diffeomorphism
\[ \Phi_k(w)= \left(f_{H, 1/k}^{kM_r} \circ f_{V, 1/k}^{kN_r} \circ \cdots  \circ f_{H, 1/k}^{kM_1}\circ  f_{V, 1/k}^{kN_1} \right) \times \id .\]
In Section \ref{action-cz} we prove the following fact, from which Theorem \ref{cone} follows.
\begin{theorem}\label{norm} Let $w \in \bF \langle V, H \rangle$ be any nontrivial word. Then there exist constants $C=C(w)>0$ and $k_0=k_0(w)$ such that \[  k> k_0 \quad \text{implies} \quad  \| \Phi_{k}(w) \|_{H}  \geq Ck.\]
\end{theorem}
\begin{remark}\label{tau} We shall see in the course of the proof of Theorem \ref{norm} that for a long word $w$, the constant  $C(w)$ is proportional to the quantity $\tau(w)$ defined in Section \ref{introduction}. Hence Theorem \ref{bounds} also follows.
\end{remark}

\section{Preliminaries on Floer theory}\label{Floer}
\subsection{Floer theory in a general class}
Let $(M, \omega)$ be a symplectically aspherical closed manifold. Pick a free homotopy class $\alpha$ in $\pi_0(\mathcal{L}M)$, and denote by $\mathcal{L}_{\alpha}M$ the path-component of $\mathcal{L}M$ containing $\alpha$. We call $M$ $\alpha$-atoroidal if for any loop $\rho$ in $\mathcal{L}_{\alpha}M$, where $\rho$ is considered as a map from $T^2$ to $M$, we have
\[ \int_{T^2}\rho^*\omega = 0 \quad \text{and}  \quad \int_{T^2}\rho^*c_1 = 0, \]
where $c_1$ denotes the first Chern class of $(M,\omega)$. Let $H: S^1 \times M \rightarrow \bR$ be a smooth function. We use the convention that the corresponding Hamiltonian vector field $X_{H}$ is defined by the equation
\[  \iota_{X_H(t)}\omega= -dH_t, \]
where $H_t(p)=H(t,p)$. A $1-$periodic orbit of the corresponding Hamiltonian flow $\varphi_H^t$ is a map $x: S^1 \rightarrow M$ satisfying
\[ \frac{d}{dt} x(t) = X_H(t, x(t)). \]
We denote by $x_0=x(0)$ the corresponding fixed point of the time-one flow $\varphi_H^1$, where $S^1=\bR / \bZ$. Suppose that $M$ is $\alpha$-atoroidal, and let $\mathcal{P}_{\alpha}(H)$ denote the set of 1-periodic orbits $x(t)$ of $H$ in the class $\alpha$.  Assuming that $H$ is $\alpha$-non-degenerate (i.e. that for each solution $x(t) \in \mathcal{P}_{\alpha}(H)$ the linearization of the time-one flow $D\phi^1_H(x):T_{x_0}M \to T_{x_0}M$ does not have 1 as an eigenvalue), the set $\mathcal{P}_{\alpha}(H)$ is finite.  As a vector space over $\bZ / 2 \bZ$, the Floer chain complex $CF_{*}(H)_{\alpha}$ associated with $H$ and $\alpha$ is generated by the elements of $\mathcal{P}_{\alpha}(H)$.

Fix a choice of representative $\eta_{\alpha} \in \alpha$ and a trivialization of the symplectic bundle $\eta_{\alpha}^*TM$. The grading on $CF_*(H)_\alpha$ is given by the Conley-Zehnder index relative to these choices, which is well defined by virtue of $M$ being $\alpha$-atoroidal. We choose the convention such that for $C^2-$small Hamiltonians the Conley-Zehnder index agrees with the Morse index. The reader interested in further details can consult \cite{Polterovich2015}. The differential on $CF_*(H)_{\alpha}$ is given by counting solutions $u=u(s,t): \bR \times S^1 \rightarrow M$ to the Floer equation

\begin{equation}\label{floer} \partial_s(u) + J(t,u)(\partial_t(u) - X_H(t,u)) = 0
\end{equation}

where $J$ is a generic, time-dependent, $1-$periodic, $\omega$-compatible almost complex structure.  To be more precise, let $x(t)$ and $y(t)$ be two periodic orbits of index $k$ and $k-1$, respectively and let $n(x,y)$ be the number (mod $2$) of solutions $u$ of (\ref{floer}) connecting $x$ to $y$.  It is well known that the number of solutions (up to a reparametrization in the $s$ factor) is finite, so that $n(x,y)$ is well-defined.  The boundary operator for our chain complex is defined as

$$\partial_{H, \alpha}(x) = \sum_{y \in CF_{k-1}(H)_{\alpha}} n(x,y) y.$$

It well known that for a generic $\omega-$compatible almost complex structure $J$, we have $\partial_{H, \alpha} \circ \partial_{H, \alpha}= 0$ so that $CF_{*}(H)_{\alpha}$ is indeed a chain complex and its homology (denoted by $HF_{*}(M)_{\alpha}$ and referred to as the $total$ Floer homology) is an invariant of the manifold $(M, \omega)$ and the class $\alpha$. In particular, $HF_{*}(M)_{\alpha}$  is independent of the Hamiltonian function $H$. The following fact is then immediate by considering a small perturbation of the (degenerate) Hamiltonian $H=0$.

\begin{proposition}  Let $\alpha$ be any non-contractible class, and assume $(M,\omega)$ is $\alpha$-atoroidal.  Then we have $HF_{*}(M)_{\alpha} = 0$.
\end{proposition}

A filtration for the chain complex is given by the action functional $\mathcal{A}_H$ from $\mathcal LM$ to $\R$. Define the action of a periodic orbit $x \in \mathcal{P}_{\alpha}(H)$ by

$$\mathcal{A}_H(x) = \int^1_0H(t,x(t)) dt - \int_{\bar{x}}\omega,$$

where $\bar{x}:S^1 \times [0,1] \to M$ is a homotopy from the fixed representative $\eta_{\alpha} \in \alpha$ to $x$. Just as with the Conley-Zehnder index, $\mathcal{A}_H(x)$ is well-defined by virtue of $(M,\omega)$ being $\alpha-$atoroidal. Observe that a curve $x:S^1 \to M$ is an element of $\cP_{\alpha}(M)$ if and only if it is a critical point of $\cA_H$.

One can extend the action functional to the whole $CF_{*}(H)_{\alpha}$ by setting
\[ \mathcal{A}_H(0) = -\infty \quad \text{and} \quad \mathcal{A}_H(\sum a_j x_j) = \text{max} \{\mathcal{A}_H(x_j) \, \,  | \, \, a_j \neq 0 \}. \]

This allows us to define the filtration $CF^{(-\infty,a)}(H)_{\alpha} := \{c \in CF_*(H)_{\alpha} | \mathcal{A}_H(c) < a \}$. Standard energy-action estimates show that $\mathcal{A}_H(\partial c) < \mathcal{A}_H(c)$ for any $c \in CF_*(H)_{\alpha}$, so that the differential on $CF_*(H)_{\alpha}$ induces a differential on $CF_{*}^{(-\infty,a)}(H)_{\alpha}$. The guiding principle is that, although the total Floer homology in a non-contractible class vanishes, the filtered Floer homology does not. Hence, geometric information can be extracted by studying the filtered Floer complex. For a more detailed discussion the interested reader can consult \cite{Polterovich2015}.
\subsection{Boundary depth}
Given a class $\alpha \in \pi_0(\cL M)$ and an $\alpha$-non-degenerate Hamiltonian $H$, M. Usher defined in \cite{Usher2009} a quantity called the boundary depth, denoted by $\beta_{\alpha}(H)$, in terms of the associated filtered Floer chain complex.
The boundary depth can be defined (see \cite[page 50]{Usher2014}) as the maximal action difference between any boundary in $CF_*(H)_{\alpha}$ and its primitive of lowest action, that is
\begin{equation}\label{eq-bd-def}
\beta_{\alpha}(H) = \sup_{0\neq x \in \text{Im} (\partial_{H, \alpha})} \left( \inf_{y \in \partial_{H, \alpha}^{-1}(x)} \bigl(\cA_H(y)-\cA_H(x) \bigr)\right).
\end{equation}
In the same paper, Usher shows that the quantity $\beta_{\alpha}(H)$ is independent the auxiliary data of the complex (such as the choice of generic almost complex structure $J$). Moreover, he proves the following properties of the boundary depth:
\begin{enumerate}
\item\label{item-1} Given two Hamiltonians $H$ and $H'$ generating the same time-one flow, we have $\beta_\alpha(H) = \beta_\alpha(H')$. Therefore, we may define unambiguously $\beta_\alpha(\phi)$ for $\phi \in \Ham(M,\omega)$.
\item\label{item-2} For any Hamiltonian diffeomorphism $\phi \in \text{Ham}(M, \omega)$, we have
  \begin{equation}\label{bdydepth}\beta_{\alpha}(\phi) \leq \|\phi\|_{\text{H}}.\end{equation}
\item\label{item-3} Given two symplectic manifold $(M, \omega_M)$ and $(N, \omega_N)$, and denoting by $\alpha_0 \in \pi_0(\cL N)$ the trivial class, we have
$$
\beta_{\alpha\times \alpha_0}(f \times \id) \geq \beta_{\alpha}(f) \label{eq-bd-product}
$$
\end{enumerate}
Proofs of these properties can be found in \cite{Usher2011}. Item \ref{item-2} is an obvious consequence of Corollary 1.5, which proves the same bound for the supremum of $\beta_{\alpha}(\phi)$ over all loop-classes $\alpha$. Finally, item \ref{item-3} follows from Theorem 8.5(b) (using the standard fact that the Hamiltonian Floer homology of any Hamiltonian diffeomorphism in the class $\alpha_0$ is non-zero).



\subsection{Spectral invariants in Floer homology}
Here we simply recall some properties of spectral invariants in Hamiltonian Floer homology (see \cite{Oh2005, Schwarz2000}) in our setting. Suppose that $(N, \Omega)$ is a closed, symplectically aspherical manifold. Then we have well defined invariants $c(a,\phi)$ for $0\neq a \in H_*(N)$ and $\phi \in \Ham(N)$, which are $1$-Lipschitz in the Hofer metric and satisfy the spectrality property $c(a,\phi) \in \spec (\phi)$ for \emph{any} $\phi \in \Ham(N)$, where $\spec(\phi)$ is the action spectrum of $\phi$. Of special use to use will be the invariants
$$
c_+ (\phi ): = c([N], \phi) \quad \text{and} \quad c_-(\phi) :=  c([\text{point}], \phi).
$$
They satisfy the following non-degeneracy property:
\begin{equation}\label{eq-nu-positive}
c_+(\phi) - c_-( \phi) > 0.
\end{equation}
Proofs of these properties, as well as further details, can be found in \cite{Schwarz2000}.

\section{Fixed points of egg-beater maps}\label{fixed}

In this section we find all the fixed points of the egg-beater map in certain carefully chosen free homotopy classes. Throughout this section, we work on a symplectic surface $(\Sigma, \omega)$ of genus $g \geq 4$ (that is, assume that $M$ is a point).

\subsection{Words and homotopy classes}
Recall from Section \ref{setup} that our local model $C \subset \Sigma$ was obtained by gluing two annuli $C_H$ and $C_V$ along the symplectomorphisms
\[ VH_0:S_{V,0} \to S_{H,0} , \quad VH_1:S_{V,1} \to S_{H,1}.\]
By assumption the maps $\pi_1(C) \to \pi_1(\Sigma)$ and $\pi_0(\cL C) \to \pi_0(\cL \Sigma)$ are injective. This will allow us to carry out our topological analysis in the local model $C$ instead of in the surface $\Sigma$ without losing any information. By abusing notation, we will go back and forth between $C$ and $\Sigma$ whenever this is convenient.
Let $A \subset C$ be the region corresponding to the squares $S_{V,0}, S_{H,0}$ and let $B \subset C$ be the region corresponding to the squares $S_{V,1}, S_{H,1}$. Denote by $\gamma_1,\gamma_2, \gamma_3,\gamma_4$ the paths in $C$ (see figure \ref{fig-gammas}) given by
\[\gamma_1 : [0,L/2] \to C_V, \quad t \mapsto (0,[t] ), \quad \gamma_2 : [0,L/2] \to C_V, \quad t \mapsto (0,[t+L/2] ) \]
\[\gamma_3 : [0,L/2] \to C_H, \quad t \mapsto (0,[t] ), \quad \gamma_4 : [0,L/2] \to C_H, \quad t \mapsto (0,[t+L/2] ). \]
Consider the concatenations $\gamma_a=\gamma_1 \# \gamma_2$, $\gamma_b=\gamma_3 \# \gamma_4$ and $\gamma_c=\gamma_3 \# \gamma_2$. These are loops in $C$ based at $(0,[0]) \in A$. Denote by $a,b,c$ the corresponding homotopy classes in $\pi_1(C,A)$. Here and below we use the contractible region $A$ as our basepoint for the fundamental group.

\begin{figure}[h]
\centering
\includegraphics[scale=1.0]{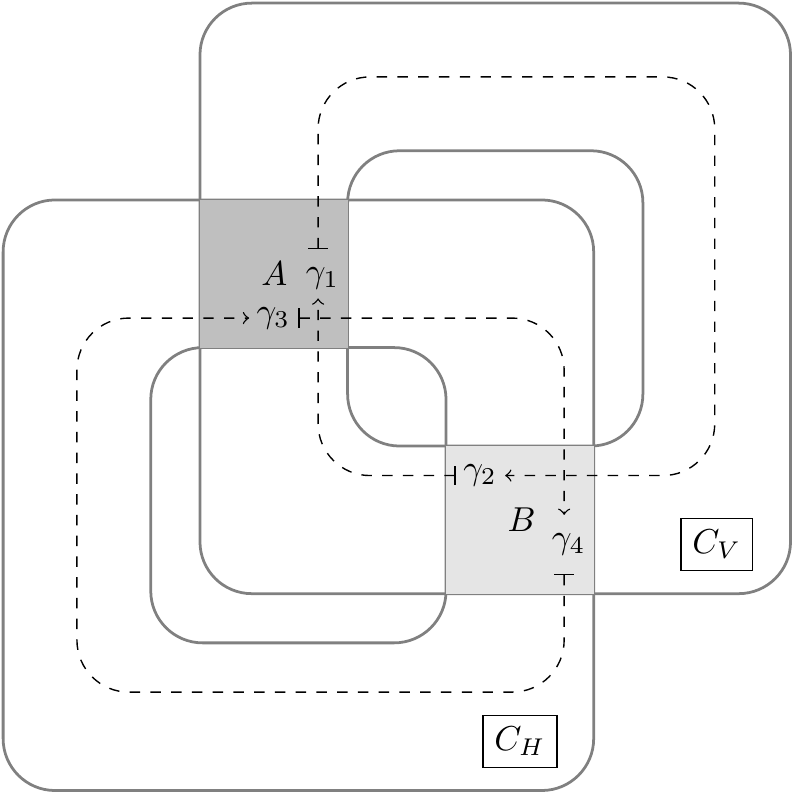}
\caption{The paths $\gamma_i$}
\label{fig-gammas}
\end{figure}

By contracting $C$ onto a graph we see that  $\pi_1(C,A) $ is isomorphic to $\bF \langle a, b, c \rangle$. Recall that the natural map $\pi_1(C,A) \to \pi_0(\cL C)$ is surjective and establishes a bijection between conjugacy classes in $\pi_1(C,A)$ and elements of $\pi_0(\cL C)$. Informally, we write $\pi_0(\cL C) \sim \pi_1(C,A)/\text{conj}$. Therefore we will think of homotopy classes based in $A$ as words in $a,b,c$ and free homotopy classes as words in $a,b,c$ modulo conjugation.

We call a reduced word $w \in \bF_2 \langle V,H \rangle$ \emph{balanced} if it has the form
\begin{equation}\label{eq-word-bal}
  w =  H^{M_r} V^{N_r} \cdots  H^{M_1} V^{N_1},
\end{equation}
with $r\geq 1$ and $N_j \neq 0$, $M_j \neq 0$ for all $1\leq j \leq r$. Observe that any word in $\mathbb{F} \langle V, H \rangle$ which is long in the sense of Definition \ref{long} is conjugate to a balanced word. By the conjugation invariance of the Hofer norm this implies that we may assume that any long word is balanced.

Recall that a word in a free group is called \emph{cyclically reduced} if all of its cyclic permutations are reduced, or equivalently, if it is reduced and does not begin and end with a letter and its inverse. It is well-known (see e.g. \cite{Cohen1989})  that conjugacy classes in a free group are classified by cyclically reduced words, up to cyclic permutations.

For the remainder of Section \ref{fixed} and throughout Section \ref{action-cz} we fix a balanced word $w \in \bF \langle V,H \rangle$ as above.

\begin{definition}\label{compatible}We call a based homotopy class $\talpha \in \pi_1(C,A) \simeq \bF \langle a,b,c, \rangle$ \emph{compatible} with $w$ if it is a cyclically reduced word of the form
  \begin{equation}\label{eq-compatible-class}
    \talpha = a^{n_1} b^{m_1} \cdots a^{n_r} b^{m_r},
  \end{equation}
  where $\sign(m_j) = \sign(M_j)$ and $\sign(n_j) = \sign(N_j)$ for all $1 \leq j \leq r$. We call a free homotopy class $\alpha \in \pi_0(\cL C) \simeq \bF \langle a,b,c, \rangle / \text{conj}$ \emph{compatible} with $w$ if it admits some representative $\talpha$ as above compatible with $w$.
\end{definition}

\subsection{Calculation of the fixed points}

\label{subsecFindingFixedPts}
Recall our egg-beater homomorphisms $\Phi_k$ defined in Section \ref{sec-eb-morphism}. To study the Floer complex of $\Phi_k(w)$ in a compatible free homotopy class $\alpha$, we first need to find all the fixed points of $\Phi_k(w)$ in the class $\alpha $. We will achieve this by finding the fixed points of a piecewise-linear version $\phi = \phi_{k,w}$ of $\Phi_k(w)$ instead. Later we will prove that for $k$ large $\phi$ and $\Phi_k(w)$ have precisely the same fixed points. Explicitly, let $f_V=f_{V,0}$ and $f_H=f_{H,0}$ be the piecewise-linear egg beater maps about $C_V$ and $C_H$. We define
\[ \phi = f_H^{M_r} \circ f_V^{N_r} \circ \cdots \circ f_H^{M_1} \circ f_V^{N_1}, \qquad \text{where} \, \, \, w=H^{M_r} V^{N_r} \cdots  H^{M_1} V^{N_1}. \]

Observe that $\Phi_k(w)$ is a smoothing of $\phi$. We will study the dynamics of the following explicit piecewise-linear Hamiltonian isotopy $\phi^t=\phi^t_{k,w}$ generating  $\phi=\phi_{k,w}$, where the pound sign means concatenation of isotopies.
\begin{equation}
  \phi^t:= \{f_V^{N_1 t}\}\# \{f_H^{M_1 t} \circ f_V^{N_1} \} \# \cdots \# \{f_H^{M_r t} \circ f_V^{N_r} \cdots \circ f_V^{ N_1} \}.\label{eq-phi^t}
\end{equation}
Our goal is to find fixed points of $\phi$ in a fixed compatible free homotopy class $\alpha$, to be specified later. That is, we seek solutions to the following system of equations.
\begin{align}
  &\phi(z) = z,\label{eq-fix-1}\\
  &[\{\phi^t(z)\}]_{\pi_0(\cL C)} = \alpha  \label{eq-hom-free}.
\end{align}

We begin by translating equation \eqref{eq-hom-free} from $\pi_0(\cL C)$ to $\pi_1(C,A)$. Given a fixed point $z$, we call the paths
$$
\{f_V^{N_1 t}(z)\},\, \{f_H^{M_1 t} \circ f_V^{N_1}(z) )\},\, \ldots , \{f_H^{M_r t}  \circ f_V^{N_r} \circ \cdots \circ f_V^{ N_1} (z)\}
$$
the \emph{intermediate paths} of $\{\phi^t (z)\}$, and we call the endpoints $z_1, \ldots, z_{2r}=z$ of the intermediate paths the \emph{intermediate points}. Our first observation is that all the fixed points and their intermediate points corresponding to a solution of \eqref{eq-fix-1} and \eqref{eq-hom-free} lie in $A$.

\begin{lemma}\label{lem-inter-A}
  Given a fixed point $z$ of $\phi$ in a compatible class $\alpha \in \pi_0(\cL C)$, all of its intermediate points lie in the region $A$. Moreover, the based homotopy class $[\{\phi^t(z)\}] \in \pi_1(C,A)$ is compatible with $w$.
\end{lemma}
\begin{proof}
  Let $z \in C$ be a fixed point of $\phi$ such that $[\{\phi^t(z)\}]_{\pi_0(\cL C)} = \alpha$. Recall that $a=[\gamma_1 \# \gamma_2]$, $b=[\gamma_3 \# \gamma_4]$ and $c=[\gamma_3 \# \gamma_2]$ freely generate $\pi(C,A) \simeq \bF \langle a,b,c \rangle$. For each of the intermediate points $z_0,z_1, \ldots , z_{2r-1}$ of the path $\{ \phi^t(z)\}$, choose a path $\beta_i$ from $z_i$ to the basepoint $(0,[0]) \in A$. More precisely, we choose the paths $\beta_i$ so that (see figure \ref{fig-betas}),
  \begin{itemize}
  \item If $z_i \in C_H \setminus C_V$ then $ \beta_i([0,1]) \subset C_H \setminus B$.
  \item If  $z_i \in C_V \setminus C_H$ then $\beta_i([0,1]) \subset C_V \setminus B$.
  \item If $z_i \in A$ then $\beta_i([0,1]) \subset A$.
  \item If $z_i \in B$ then $\beta_i \sim \gamma_2$ as paths $([0,1],0,1) \to (C,B,A)$.
  \end{itemize}
  
\begin{figure}[h]
\centering
\includegraphics[scale=1.0]{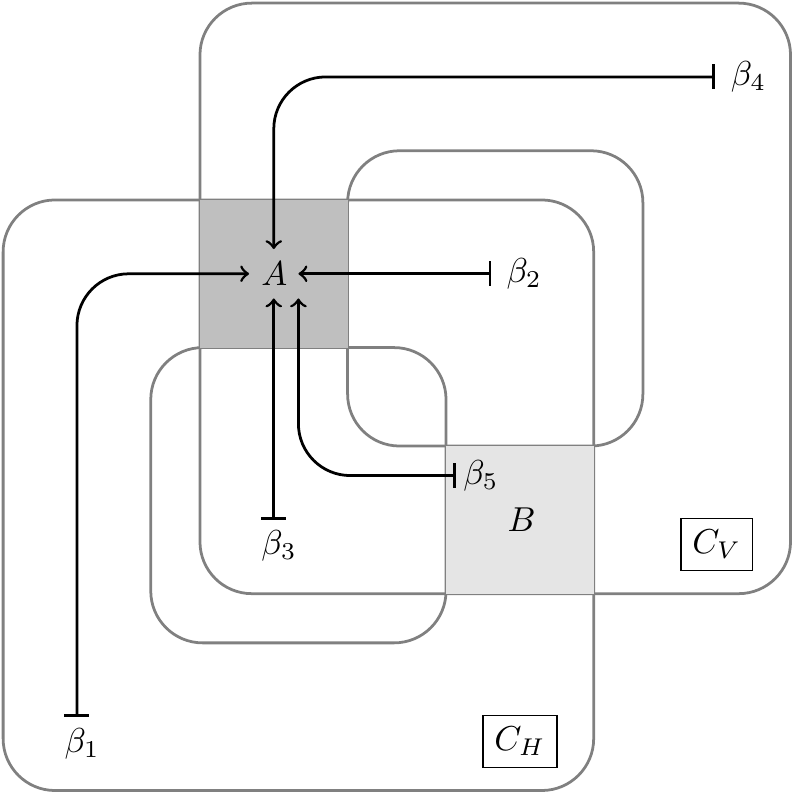}
\caption{The paths $\beta_i$ for the different possible starting points}
\label{fig-betas}
\end{figure} 
  
  For every intermediate path $\{ f_V^{N_j t}(z_{2j-2}) \}$, or $\{ f_H^{M_j t}(z_{2j-1}) \}$ we associate an element in $\pi_1(C,A)$,
  \[l_{2j-2}=[\beta_{2j-2}^{-1}\# \{f_V^{N_j t}(z_{2j-2})\}\# \beta_{2j-1}],\]
  \[ l_{2j-1} = [ \beta_{2j-1}^{-1}\#\{ f_H^{M_j t}(z_{2j-1})\}\# \beta_{2j}] .\]
  Thus we express the class $[ \phi^t(z)] \in  \pi_0(\cL C)$ as the conjugacy class of the product $l_0 \cdots l_{2r-1} \in \pi_1(C,A)$, where we have
  \[ l_{2j-2} = a^{n_j},\]
  \[ l_{2j-1} = c^{-\varepsilon_j} b^{m_j} c^{\nu_j} ,\]
  for some $n_j,m_j \in \bZ$, and $\varepsilon_j,\nu_j \in \{0,1\}$. Note the following observations:
  \begin{itemize}
  \item $n_j \neq 0$ implies $\sign(n_j)=\sign(N_j)$.
  \item $m_j \neq 0$ implies $\sign(m_j)=\sign(M_j)$.
  \item $z_{2j-1} \not\in B$ if and only if $\varepsilon_j=0$.
  \item $z_{2j} \not \in B$ if and only if $\nu_j=0$.
  \item $z_{2j-2} \notin C_V  $ implies $n_j=0$. \item $z_{2j-1} \notin C_H$ implies $m_j=0$. \item $z_{2j} \notin C_H$ implies $m_{j}=0$.  \item$z_{2j-1} \notin C_V$ implies $n_{j} =0$.
  \end{itemize}
  By hypothesis, $[\{ \phi^t(z)\}]_{\pi_0(\cL C) }$ is compatible with $w$. Therefore the product
  \[  a^{n_1}c^{\varepsilon_1}b^{m_1}c^{\nu_1}\ldots a^{n_r}c^{\varepsilon_r}b^{m_r}c^{\nu_r} \]
  is in the same conjugacy class as some word in $\mathbb{F} \langle a,b,c \rangle$ compatible with $w$. Since conjugacy classes are classified by cyclically reduced words up to cyclic permutation, we must have $\varepsilon_j=\nu_j=0$ and $n_j,m_j \neq 0$ for all $j$. The above observations imply that $z_j \in C_H \cap C_V$ and $z_j \notin B$, hence $z_j \in A$ for all $j$. Moreover, $[\phi^t(z_0)]_{\pi_1(C,A)}$ is compatible with $w$. This completes the proof. \end{proof}

Let $\widetilde{\alpha}=a^{n_1}b^{m_1} \cdots a^{n_r} b^{m_r}$ be any given representative of a compatible class $\alpha \in \pi_0(\cL)$. Suppose that $[\{\phi^t(z)\}]_{\pi_0(\cL C)}=\alpha$. Then, the based homotopy class $[\{\phi^t (z)\}]_{\pi_1(C,A)}$ must be one of the classes
$$
\widetilde{\alpha}_j = a^{n_{j+1}}b^{m_{j+1}} \cdots a^{n_r}b^{m_r} a^{n_1} b^{m_1} \cdots a^{n_{j}} b^{m_{j}}, \quad 0\leq j\leq r-1.
$$
It follows that equations \eqref{eq-fix-1} and \eqref{eq-hom-free} are equivalent to the following systems of equations, indexed by $0\leq j \leq r-1$,
\begin{align}
  &\phi(z) = z, \label{eq-fix-2} \\
  &[{\phi^t (z)}]_{\pi_1(C,A)} = \talpha_j, \label{eq-hom-based}
\end{align}
such that $\talpha_j$ is compatible with $w$. Fix now an integer $k \in \Z$. We define a special compatible homotopy class $\talpha(k,w) \in \pi_1(C,A)$ by setting
\begin{equation}\label{eq-mn}
  n_j = k \sign(N_j),  \quad m_j = k \sign(M_j)
\end{equation}
in \eqref{eq-compatible-class}. We denote by $\alpha(k,\omega) \in \pi_0(\cL \Sigma)$ the free homotopy class represented by $\talpha(k,w)$. By the special choice of $\alpha=\alpha(k, w)$, all compatible $\talpha_j$ coincide with $\talpha(k,w)$. We are therefore left with solving the following system of equations.
\begin{align}
  &\phi(z) = z, \label{eq-fix-3} \\
  &[{\phi^t( z)}]_{\pi_1(C,A)} = \talpha(k,w). \label{eq-hom-talpha}
\end{align}

This concludes the reduction from $\pi_0(\cL C)$ to $\pi_1(C,A)$.

\begin{theorem}\label{thm-fixed-pts}
  Given a balanced word $w = H^{M_r}V^{N_r} \cdots H^{M_1}V^{N_1}$, for large enough $k \in \Z$, the system of equations $[$\eqref{eq-fix-3}, \eqref{eq-hom-talpha}$]$
  admits exactly $2^{2r}$ solutions. Each of the $2^{2r}$ fixed points is non-degenerate.
\end{theorem}

Our starting observation is that by Lemma \ref{lem-inter-A}, the intermediate paths have well-defined classes in $\pi_1(C,A)$, which satisfy
\[
  [\{f_V^{N_1 t}(z)\}] = a^{n_1}, \, \, \, [\{f_H^{M_1 t}\circ f_V^{N_1} (z)\}] = b^{m_1}, \cdots \label{eq-inter-htpy} \] \[
  \cdots , [\{f_V^{N_rt}\circ f_H^{M_{r-1}}\circ  \cdots \circ f_V^{N_1}(z)\}]=a^{n_r}, \, \,  \, [\{f_H^{M_{r}t}\circ f_V^{N_{r}} \circ  \cdots \circ  f_V^{N_1}(z)\}]=b^{m_r}. \]
With this in mind, we translate the sequence of intermediate points to the universal cover $[-1,1] \times \bR$ of $C_0$, and after a shift we write them explicitly as points in $(-1,1)\times (-1,1)$. Our convention is that we write the even intermediate points in terms of the coordinates on $C_V$, and the odd ones in terms of the coordinates on $C_H$.

We consider, for $p\in\Z$, the map
$$
r_p \colon (-1,1) \times \R \to (-1,1) \times \R,\quad (x,y) \mapsto (x,y-p\cdot L).
$$
Observe that $r_p$ leaves invariant the projection $(-1,1) \times \R \to (-1,1) \times \R / L\Z$. Recall the map $VH_0:S_{0,V} \to S_{0,H}$, which we now think of as a map
$$VH \colon (-1,1) \times (-1,1) \to (-1,1) \times (-1,1), \,\, \, \, (x,y) \mapsto (-y,x).$$
We will also use its inverse $VH$, given by
$$ HV \colon (-1,1) \times (-1,1) \to (-1,1) \times (-1,1),\,\, \, \, (x,y)\mapsto (y,-x).
$$
Finally, we denote by $\widetilde{f_0}$ the lift of the piecewise-linear shear map $f_0:C_0 \to C_0$ to the universal cover $[-1,1]\times \R$ of $C_0$ given by the time-$1$ map of the lift of the isotopy $\{f^t_0\}$. Explicitly, $\tf_0$ can be written as the following piece-wise affine map.
\begin{equation}\label{eq-tf-matrix}
  \tf_0(x,y) = \begin{pmatrix}
    1 & 0 \\ -2 L  \eps_x & 1
  \end{pmatrix}
  \begin{pmatrix}
    x \\ y
  \end{pmatrix} + \begin{pmatrix}
    0 \\ 2L
  \end{pmatrix},
\end{equation}
where $\eps_x = \sign(x)$. With this language, if we write $\widetilde{z}_j=(x_j, y_j)$, $0\leq j\leq 2r$, for the lift of the intermediate points to $(-1,1) \times (-1,1)$, then the fixed point equation reads $(x_{2r},y_{2r})=(x_0,y_0)$, where
\begin{align*}
  (x_{2j+1},y_{2j+1}) &= VH \circ r_{n_{j+1}} \circ \widetilde{f}_0^{\, kN_{j+1}} \,(x_{2j},y_{2j}),\\
  (x_{2j+2}, y_{2j+2}) &= HV \circ r_{m_{j+1}} \circ \widetilde{f}_0^{\, kM_{j+1}} \,(x_{2j+1},y_{2j+1}).
\end{align*}
Restricting attention to the even intermediate points, we write, for $1 \leq j \leq 2r$,
\begin{equation*}
  \Theta_j = HV \circ r_{m_{j}} \circ \widetilde{f}_0^{\, kM_j} \circ  VH \circ r_{n_j} \circ \widetilde{f}_0^{\, kN_j}.
\end{equation*}
so that the fixed point equation becomes
\begin{align}
  (x_{2j+2},y_{2j+2}) &= \Theta_{j+1} (x_{2j}, y_{2j}), \,\, 0\leq j \leq r-1,\label{eq-inter-coords}\\
  (x_{2r},y_{2r}) &= (x_0,y_0)\label{eq-fix-coords}.
\end{align}
Note that for any $(x,y)\in (-1,1) \times (-1,1)$ satisfying
\begin{align}
  \widetilde{f}_0^{\, kN_{j}}(x,y)  &\in (-1,1) \times (n_{j}L-1,n_{j+1}L+1)\label{eq-cond-Nj}  \\
  \widetilde{f}_0^{\, kM_{j}} \circ  VH \circ r_{n_{j}} \circ \widetilde{f}_0^{\, kN_{j}} (x,y)  &\in (-1,1) \times (r_{m_{j}}L-1,r_{m_{j}}L+1) \label{eq-cond-Mj}
\end{align}
the composition $\Theta_{j} (x,y)$ is well-defined. Note that the lifts of the intermediate points satisfy conditions \eqref{eq-cond-Nj} and \eqref{eq-cond-Mj} in view of condition \eqref{eq-hom-talpha} on the homotopy class of the trajectory.
Note also that the odd intermediate points are determined by the even ones, as one has
\begin{equation}\label{eq-even-odd}
  (x_{2j+1},y_{2j+1})=(-y_{2j+2},x_{2j}).
\end{equation}
We can simplify equation \eqref{eq-inter-coords} if we denote $\eps_{j} = \sign(x_{j})$, for $0\leq j \leq 2r-1$. Define additionally
$$
\nu_j = \frac{n_j}{2 N_j k}, \quad \mu_j = \frac{m_j}{2 M_j k}.
$$
After plugging in our special choices for $n_j$ and $m_j$ defined by \eqref{eq-mn}, we get
\begin{equation}\label{eq-munu}
\nu_j = \frac{1}{2|N_j|}, \quad \mu_j = \frac{1}{2|M_j|}.
\end{equation}
We can then simplify the right hand side of \eqref{eq-inter-coords} to:
\begin{equation*}
  \Theta_{j+1}(x_{2j},y_{2j}) = A_{j+1} (x_{2j},y_{2j}) + v_{j+1},
\end{equation*}
where
$$
A_{j} = \begin{pmatrix}
  1-4L^2k^2 N_j M_j \eps_{2j-2} \eps_{2j-1} & 2L k M_j \eps_{2j-1} \\ -2L k N_j \eps_{2j-2} & 1
\end{pmatrix},
$$
and \begin{equation}\label{eq-vj}
  v_{j} = (4L^2k^2 \eps_{2j-1}M_j N_j(1-\nu_j) + 2L k M_j (1-\mu_j), 2L k N_j (1-\nu_j)).
\end{equation}
Observe that $\det A_j = 1$. Moreover, we note for future reference that
\begin{equation}\label{eq-Ajvj}
  A_j^{-1}v_j = (2L k M_j (1-\mu_j), 4L^2k^2 \eps_{2j-2}M_j N_j (1-\mu_j) + 2L k N_j (1-\nu_j)).
\end{equation}
With these notations, equation \eqref{eq-fix-coords} becomes
\begin{equation}\label{eq-fixedpts-matrix}
  (\overline{A} - \id)(x_0,y_0) =  - \overline{v},
\end{equation}
where $\overline{A} = A_{r} \cdot A_{r-1} \cdot \ldots \cdot A_1$ and
\begin{equation}\label{eq-vbar}
  \overline{v} = v_{r} + \sum_{j=1}^{r-1} A_r \cdot \ldots \cdot A_{j+1} \cdot v_j.
\end{equation}

We will need to know the behaviour of $\overline{A}$ for large $k$. For this we denote $\overline{A}_{i,j} = A_j \cdot A_{j-1} \cdot \ldots \cdot A_i$, and $\overline{A}_j = A_{1,j}$.
\begin{lemma}\label{lem-Abarj-asym}
  Denote $\alpha_j =-2L N_j \eps_{2j+2}$, $\beta_j =  2L M_j \eps_{2j+1} $ and for $i\leq j$, $\gamma_{i,j} = \beta_j \alpha_j ... \beta_i \alpha_i$. Then for any $1\leq i < j \leq r$ one has
  \begin{equation}\label{eq-Abar-asym}
    \overline{A}_{i,j} = \begin{pmatrix}
      k^{2(j-i+1)}\gamma_{i,j} & k^{2(j-i)+1} \frac{\gamma_{i,j}}{\alpha_i} \\ k^{2(j-i)+1} \frac{\gamma_{i,j}}{\beta_j} & -k^{2(j-i)}\frac{\gamma_{i,j}} {\alpha_i \beta_j}
    \end{pmatrix} \cdot \left( 1 + O\left(\frac{1}{k}\right)\right), \quad k \to \infty,
  \end{equation}
  To be precise, in the equation above we mean that each entry of $\overline{A}_{i,j}$ is multiplied by a different term of the form $1+O(1/k)$.
\end{lemma}
The proof of Lemma \ref{lem-Abarj-asym} is straightforward induction. With this in hand, we can proceed with the proof of the theorem.
\begin{proof}[Proof of Theorem \ref{thm-fixed-pts}]
  We claim that, for large enough $k$, we can find, for each $\vec{\eps} = (\eps_0, \ldots, \eps_{2r-1}) \in \{\pm 1\}^{2r}$ a solution to equations \eqref{eq-inter-coords}, \eqref{eq-fix-coords} such that $\sign(x_j) = \eps_j$ for all $0 \leq j \leq 2r-1$. Indeed, suppose $\{(x_{2j},y_{2j})\}_{j=0}^{r-1}$ is a solution with sign vector $\vec{\eps}$. As noted above, equation \eqref{eq-fix-coords} can be rewritten as \eqref{eq-fixedpts-matrix}. We claim that, for large enough $k$, the matrix $\overline{A}-\id$ is nonsingular. For this we use the fact that for every $2 \times 2$ matrix $A$ with $\det A = 1$ one has $\det(A-\id) = 2-\trace(A)$. Using this, and the asymptotic expansion for $\overline{A}=\overline{A}_r$ provided by Lemma \ref{lem-Abarj-asym}, we compute
  \begin{equation}\label{eq-det-asym}
    \det(\overline{A}-\id) = 2-\trace(\overline{A}) = -k^{2r} \gamma_{1,r} (1+O(1/k)), \quad k \to \infty.
  \end{equation}
  In particular, for $k$ large enough, this determinant is non-zero, and hence equation \eqref{eq-fixedpts-matrix} has a unique solution $(x_0,y_0)$. Given the solution $(x_0,y_0)$, we form the sequence of intermediate points $\{(x_j,y_j)\}_{j=0}^{2r-1}$ using \eqref{eq-inter-coords}. We claim that they satisfy the following asymptotic form as $k \to \infty$, for all $0\leq j \leq r-1$:
  \begin{equation}\label{eq-inter-pts-asym}
    (x_{2j}, y_{2j}) = (\eps_{2j} (1-\nu_{j+1}), -\eps_{2j-1} (1-\mu_{j})) + O(1/k).
  \end{equation}
  In these equations the index of $\eps_j$ is considered modulo $2r$, and the indices of $\mu_j$ are $\nu_j$ are considered modulo $p$. Postponing the proof of these asymptotics for the moment, we deduce that the sequence $\{(x_{2j},y_{2j} ) \}$ satisfy \eqref{eq-cond-Nj} and \eqref{eq-cond-Mj} and conclude that the full sequence $\{ (x_j,y_j) \}$ defines lifts of intermediate points corresponding to a fixed point $(x_0,y_0)$ solving equations \eqref{eq-fix-3} and \eqref{eq-hom-talpha}. This proves that for any $\vec{\eps} \in \{\pm 1\}^{2r}$ there exists a unique solution, as asserted. Next, let us show that these fixed points are all non degenerate. Indeed, as a piece-wise affine map, we have
  $$
  D_z\phi - \id = \overline{A} -\id,
  $$
  which we have seen to be non-degenerate for large enough $k$.

  Therefore to complete the proof of Theorem \ref{thm-fixed-pts} it remains to prove the asymptotic expansion \eqref{eq-inter-pts-asym}. We begin with proving it for $j=0$. By \eqref{eq-fixedpts-matrix},
  \begin{equation}\label{eq-x0y0}
    (x_0,y_0) = -(\overline{A}-\id)^{-1} \overline{v}
  \end{equation}
  We note that if $A$ is a $2 \times 2$ matrix with $\det A = 1$ and $(A-\id)$ invertible, one has $(A-\id)^{-1} = \det (A - \id)^{-1} \cdot (A^{-1}-\id)$. Applying this to $\overline{A}$, we rewrite \eqref{eq-x0y0} using \eqref{eq-vbar} as
  \begin{equation}\label{eq-coords-eq-0}
    (x_0,y_0) = - \det(\overline{A}-\id)^{-1} \cdot \left[ \overline{A}v_r- v_r + \sum_{j=1}^{r-1} (\overline{A}_j^{-1} v_j - \overline{A}_{j+1,r}v_j)\right].
  \end{equation}
  Let us analyze each of the $r$ summands of the right hand side of \eqref{eq-coords-eq-0}. We begin with the summand corresponding to $v_j$, for $1<j<r$. It has the form
  $$
  -\,\frac{\overline{A}_j^{-1}v_j-\overline{A}_{j+1,r}v_j}{\det(\overline{A}-\id)}
  $$
  We rewrite the first term in the nominator as
  \begin{equation}\label{eq-summand_j}
    \overline{A}_j^{-1}v_j = \overline{A}_{j-1}^{-1}\left(A_j^{-1}v_j\right).
  \end{equation}
  Now, using \eqref{eq-Ajvj} and \eqref{eq-Abar-asym} we see that all entries of $\overline{A}_{j-1}$ are $O(k^{2r-2})$ (recall that $j\leq r-1$) and those of $A_j^{-1}v_j$ are $O(k^2)$. Similarly, estimating $\overline{A}_{j+1,r}$ using \eqref{eq-Abar-asym} and $v_j$ by its definition \eqref{eq-vj} we deduce that the entries of the nominator of \eqref{eq-summand_j} are $O(k^{2r-2})$. Finally, using \eqref{eq-det-asym} to estimate the denominator, we deduce that this entire summand is $O\left(1/k^2 \right)$.

  Next let us turn to the summand corresponding to $v_1$. We first observe that by \eqref{eq-Ajvj}, the entries of $\overline{A}_1^{-1}v_1 = A_1^{-1}v_1$ are all $O(k^2)$. Next, we compute the second term using \eqref{eq-Abar-asym}:
  \begin{align*}
    \overline{A}_{2,r}v_1 &= \begin{pmatrix}
      k^{2r-2} \gamma_{2,r} + O(k^{2r-3}) & O(k^{2r-3}) \\ O(k^{2r-3}) & O(k^{2r-3})
    \end{pmatrix}
                                                                         \begin{pmatrix}
                                                                           4L^2k^2 \eps_1 M_1 N_1 (1-\nu_1) + O(k) \\ O(k)
                                                                         \end{pmatrix} \\
                          &=
                            \bigl(k^{2r}4L^2\gamma_{2,r}\eps_1 M_1 N_1(1-\nu_1),0\bigr) + O(k^{2r-1})
  \end{align*}
  Again using \eqref{eq-det-asym} to estimate the denominator and recalling that by definition $\gamma_{1,r} = -\gamma_{2,r} \cdot 4L^2 M_1 N_1 \eps_{0} \eps_1$ we compute
  \begin{align*}
    -\,\frac{\overline{A}_1^{-1}v_1-\overline{A}_{2,r}v_1}{\det(\overline{A}-\id)} &=-\left( \frac{\gamma_{2,r}\eps_1 4L^2M_1 N_1}{\gamma_{1,r}} (1-\nu_1),0\right) + O(1/k) \\
                                                                                   &= (\eps_0(1-\nu_1),0) + O(1/k).
  \end{align*}
  Finally we turn to the term corresponding to $v_r$. First we observe that, as before, the entries of $v_r$ are $O(k^2)$. Next, using \eqref{eq-Abar-asym} and \eqref{eq-Ajvj}, we compute
  \begin{align*}
    \overline{A}^{-1}v_r &= \overline{A}_{r-1}^{-1} A_r^{-1}v_r \\
                         &= \begin{pmatrix}
                           O(k^{2r-3}) & O(k^{2r-3}) \\
                           O(k^{2r-3}) & k^{2r-2}\gamma_{1,r-1} + O(k^{2r-3})
                         \end{pmatrix}
                                         \begin{pmatrix}
                                           O(k) \\ 4L^2 k^2 \eps_{2r-2}M_r N_r (1-\mu_r) + O(k)\end{pmatrix} \\
                         &= \bigl(0,k^{2r}\gamma_{1,r-1} \eps_{2r-2} 4L^2 M_r N_r (1-\mu_r)\bigr) + O(k^{2r-1}).
  \end{align*}
  Finally, using as above \eqref{eq-det-asym} for the denominator, and recalling that by definition $\gamma_{1,r} = -\gamma_{1,r-1} 4L^2 M_r N_r \eps_{2r-2}\eps_{2r-1}$, we obtain
  $$
  -\,\frac{\overline{A}^{-1}v_r-v_r}{\det(\overline{A}-\id)} =  (0,-\eps_{2r-1}(1-\mu_r))+O(1/k).
  $$
  Combining all these terms, we get
  $$
  (x_0,y_0) = (\eps_0(1-\nu_1),-\eps_{2r-1}(1-\mu_r)) + O(1/k),
  $$
  which is the required asymptotic \eqref{eq-inter-pts-asym} for $j=0$ (recall our cyclic index convention). The asymptotics for other values of $j$ are proven similarly, once we observe that any even intermediate point $(x_{2j},y_{2j}$) satisfies
  $$
  (x_{2j},y_{2j}) = \Theta_{j} \circ\ldots  \circ \Theta_1 \circ \Theta_r \circ \ldots \circ \Theta_{j+1}(x_{2j},y_{2j}).
  $$
  This concludes the proof of the asymptotic formula \eqref{eq-inter-pts-asym}, and thus as explained, the proof of the theorem. \end{proof}
\section{Floer theory of egg-beater maps}\label{action-cz}
In this section we compute part of the Floer data of our egg-beater map in our chosen free homotopy class and use it to complete the proof of our main result. Since, as noted above, the proof follows from the case of a surface, we again work on a symplectic surface $(\Sigma, \omega)$ of genus $g \geq 4$, and return to the general case only in Subsection \ref{subsect-linear-growth}.

\subsection{Asymptotics for the action}
In this subsection we compute asymptotic formulas for the actions of the fixed points provided by Theorem \ref{thm-fixed-pts}. We follow the notations of the previous section: we let $z(\vec{\eps})$ be the unique fixed point of $\phi$ in the compatible class
$\alpha(k,w) \in \pi_0(\cL \Sigma)$ with sign vector $\vec{\eps}$, where we take $k$ sufficiently large. As a reference loop in the free homotopy class $\alpha(k, w) = a^{n_1} b^{m_1} \cdots a^{n_r} b^{m_r}$ we take the loop
\begin{equation}\label{eq-eta}
  \eta_\alpha = \gamma_a^{\#n_1} \# \gamma_b^{\# m_1} \# \cdots \# \gamma_a^{\#n_r} \#\gamma_b^{\#m_r}.
\end{equation}
The result of our computation is as follows:
\begin{proposition}\label{prop-action-asym}
  The action of the fixed point $z(\vec{\eps})$ has the following asymptotic formula as $k \to \infty$:
  \begin{equation}\label{eq-action}
    \cA(z(\vec{\eps})) = L k\sum_{j=1}^{r} \left[ \eps_{2j-2} N_j \left(1-\nu_j\right)^2 + \eps_{2j-1} M_j \left(1-\mu_j \right)^2\right] + O(1).
  \end{equation}
\end{proposition}
%
%
\begin{proof} 
  Recall that we are using the Hamiltonian isotopy $\{\phi^t\}$ defined by \eqref{eq-phi^t}. This isotopy is a concatenation of $2r$ Hamiltonian isotopies of shear maps of $C_V$ and $C_H$, and hence the action $\cA(z(\vec{\eps}))$ is the sum of $2r$ terms, each corresponding to an intermediate path. Using the asymptotics \eqref{eq-inter-pts-asym} we compute, for example, the first summand to be
  \begin{align*}
    \cA_1 &= \int_0^1 kN_1 h_0(f_V^t(x_0,[y_0]) ) -\int_{\bar{z}} dx\wedge dy \\
          &= k N_1 h_0(x_0,[y_0]) - x_0 2L k N_1 (1-\eps_0 x_0) + O(1) \\
          &= 2L k N_1x_0 \left[  \left(1-\frac{\eps_0 x_0}{2} \right) - (1-\eps_0 x_0) \right] +O(1) \\
          &= L k \eps_0 N_1 x_0^2 + O(1) =  L k \eps_0 N_1 \left(1-\nu_1 \right)^2  + O(1).
  \end{align*}
  The other summands are computed similarly. \end{proof}

\subsection{The Conley-Zehnder index}\label{cz}
We derive a formula expressing the Conley-Zehnder indices of the periodic orbits given by Theorem \ref{thm-fixed-pts}. We continue with the notation $z(\vec{\eps})$ for the fixed point corresponding to the choice of signs $\vec{\eps}=(\eps_0,...,\eps_{2r-1}) \in \{ \pm 1 \}^{2r}$ as described above. Recall that our choice of Hamiltonian isotopy generating $\phi = \phi_{k,w}$ is given by \eqref{eq-phi^t}. Moreover, we seek fixed points in the free homotopy class $\alpha(k,w)$, determined by \eqref{eq-mn}, and as a reference loop in this class we take the loop $\eta_\alpha$ determined by \eqref{eq-eta}. Finally, our choice of framing along $\eta_\alpha$ is as follows:
\begin{equation}
  \text{$\left( \frac{\del}{\del x}, \frac{\del}{\del y} \right)$ along $\gamma_a$,} \qquad \text{ $\left(-\frac{\del}{\del y},\frac{\del}{\del x} \right)$ along $\gamma_b$}. \label{eq-framing}
\end{equation}Observe that this choice defines a smooth framing by definition of the gluing diffeomorphism. With this preparation we state the main result of this section.
\begin{theorem}\label{thm-cz-sgn}  Let $z(\vec{\eps})$ be a fixed point of $\Phi_k(w)$ in the class $\alpha(k,w)$with associated sign vector $\vec{\eps}$.  Then for large enough $k$ one has
  \begin{equation}\label{eq-CZ}  \mu_{CZ}(z(\vec{\eps}))  =1 + \frac{1}{2} \sum_{j=1}^{r} \bigl[\eps_{2j-2} \sign (N_j)  + \eps_{2j-1}  \sign (M_j) \bigr].
  \end{equation}
\end{theorem}

This formula can be explained as follows. From the definition of the isotopy \eqref{eq-phi^t} it is easy to see that the linearized flow is a concatenation of $2r$ (degenerate) paths of symplectic shear matrices. We use the definition of the Conley-Zehnder index in terms of the Robbin-Salamon index, which is defined also for degenerate paths. Our normalization is chosen so that the Conley-Zehnder index of a periodic orbit corresponding to a critical point of a $C^2$-small Morse function equals the Morse index. Explicitly, the Conley-Zehnder index on a $2n$-dimensional manifold is $n$ minus the Robbin-Salamon index.  The index of a path of shear matrices is an easy and well-known computation, and so the bulk of the work is relating the index of the concatenation to the indices of the concatenated paths, for which we use results from \cite{DeGosson2008b}. We begin by recalling some old notation and introducing some new notation.


\begin{itemize}
\item For a fixed point $z$ of $\phi$, we denote by $
  z_0, z_1, z_2, \ldots, z_{2r-1}, z_{2r}$
  the corresponding intermediate points. Note that $z_0=z=z_{2r}$.
\item For our piecewise-linear homeomorphism
  $$
  \phi = f_H^{kM_r} \circ f_V^{kN_r} \circ \ldots \circ f_H^{kM_1} \circ f_V^{kN_1},
  $$
  denote, for $1\leq j \leq r$,
  $$
  \ph_j = f_H^{kM_j} \circ f_V^{kN_j} \circ \ldots \circ f_H^{kM_1} \circ f_V^{kN_1},\quad  \psi_{j} =  f_{H}^{kN_j}\circ \ph_{j-1}.
  $$
\end{itemize}
With this notation, the intermediate points are
$$
z_{2j} = \ph_j( z) \quad\text{and}\quad  z_{2j-1} = \psi_j (z),
$$
and  the loop $\gamma_z(t)=\phi^t(z)$ is
\begin{equation*}
  \gamma_z(t)  =
  \#_{j=1}^{r}
  \Bigl(\{f_V^{kN_jt} (z_{2j-2})\} \# \{f_H^{kM_jt} (z_{2j-1})\}  \Bigr).
\end{equation*}

Using the framing \eqref{eq-framing} we can identify the linearized flows with the following symplectic matrices.
\begin{equation}\label{eq-df-in-triv}
  df_V^{k N_j t}(p) \sim \begin{pmatrix}
    1 & 0 \\ -2LkN_j \eps_V(p)t & 1
  \end{pmatrix}, \,\,
  df_H^{k M_j t}(p) \sim
  \begin{pmatrix}
    1 & 2LkM_j \eps_H(p)t \\ 0 & 1
  \end{pmatrix}  ,
\end{equation}
where $\eps_V(p)$ is the sign of the $x$-coordinate of $p$, viewed in $V$ (and similarly for $\eps_H$). To abbreviate, we denote for $\alpha, \beta \in \bR$,
$$
A_t(\alpha) = \begin{pmatrix}
  1 & 0 \\ -\alpha t & 1
\end{pmatrix}, \quad B_t(\beta ) = \begin{pmatrix}
  1 & \beta t \\ 0 & 1
\end{pmatrix}.
$$
Here and below all paths of matrices appearing below are parametrized by $t \in [0,1]$. Denote also $A(\alpha) = A_1(\alpha)$ and $B(\beta) = B_1(\beta)$. We note that $A_t(\alpha)$ is conjugate to $B_t(\alpha)$. Thus setting  $\alpha_j = 2Lk N_j \eps_V(z_{2j-2})$, $\beta_j = 2Lk M_j \eps_H(z_{2j-1})$, using \eqref{eq-df-in-triv} we identify the linearized flow with the following path of symplectic matrices:
\begin{equation}\label{eq-lin-catenation}
  \#_{j=1}^r \Bigl( \{A_t(\alpha_j) \cdot \overline{B}_{j-1} \} \# \{B_t(\beta_j) \cdot \overline{A}_j \}\Bigr),
\end{equation}
where we set
$$
\oB_j = B(\beta_j) \cdot A(\alpha_j) \cdot \ldots \cdot B(\beta_1) \cdot A(\alpha_1), \quad \overline{A}_j = A(\alpha_j)\cdot \overline{B}_{j-1},
$$
with $\overline{B}_{0}=\id$. Denote further
$$
\Gamma_j(t) =  \{A_t(\alpha_j) \cdot \overline{B}_{j-1} \} \# \{B_t(\beta_j) \cdot \overline{A}_j \}, \quad \Psi_j(t)= \{A_t(\alpha_j) \} \# \{B_t(\beta_j) \cdot A(\alpha_j) \}.
$$
For the initial step ($j =1$),
\begin{equation} \label{j=1}
  \Gamma_1(t)  = \Psi_1(t) = \{A_t(\alpha_1) \} \# \{B_t(\beta_1) \cdot A(\alpha_1)\}
\end{equation}
and for any $j \in \bN$, we have the relation
\begin{equation} \label{inductive for j}
  \Gamma_{j+1}(t)= \Psi_{j+1}(t) \cdot \Gamma_j(1)
\end{equation}

Finally, set
$$
\overline{\Gamma}_l(t) = \#_{j=1}^l \{\Gamma_j(t)\},
$$
so that the RS-index of \eqref{eq-lin-catenation} is expressible as
\begin{equation}\label{eq-Phibar_k} i_{RS}(\{ \overline{\Gamma}_r(t) \} ).
\end{equation}
We note that $\overline{\Gamma}_l(1) = \Gamma_l(1)$. Our desired concatenation formula is the following:

\begin{lemma}\label{thm-cat-tot}  For $k$ sufficiently large,

  \begin{equation}\label{eq-Phibar_k-sum} i_{RS}(\{\overline{\Gamma}_r(t)\}) = \sum_{j=1}^r i_{RS}\bigl( \{A_t(\alpha_j)\} \bigr) + i_{RS} \bigl(\{B_t (\beta_j)\} \bigr).
  \end{equation}

\end{lemma}

Lemma \ref{thm-cat-tot} is an obvious corollary of the following two lemmas.

\begin{lemma}\label{lem-cat-1}
  If $k$ is sufficiently large, then for all $1\leq l \leq r$ one has
  $$
  i_{RS}(\{\overline{\Gamma}_l(t)\}) = \sum_{j=1}^{l} i_{RS}(\{\Psi_j(t)\}).
  $$
\end{lemma}

\begin{lemma}\label{lem-cat-2}
  If $k$ is sufficiently large, then for all $1 \leq j \leq r$,
  $$
  i_{RS}\bigl( \{\Psi_j (t) \}\bigr) = i_{RS}\bigl( \{A_t(\alpha_j)\} \bigr) + i_{RS} \bigl(\{B_t (\beta_j)\} \bigr).
  $$
\end{lemma}

\subsubsection{Proof of Lemma \ref{lem-cat-1}}

For a symplectic matrix $P$ with $\id - P$ invertible, define the \emph{Cayley transform} of $P$ to be the symmetric matrix
$$
M_P = \frac{1}{2} J (\id + P)(\id - P)^{-1}.
$$
Here $J$ is the standard complex structure on $\bR^{2n} \simeq \bC^n$. The proof of Lemma \ref{lem-cat-1} is an application of the following result regarding concatenating non-degenerate paths of matrices, which combines Corollary 3.5 and Lemma 3.3 from \cite{DeGosson2008b}.

\begin{theorem}\label{thm-PDG}
  Let $\{P_1(t)\}_{[0,1]}$ and $\{P_2(t)\}_{[0,1]}$ be two non-degenerate paths of symplectic matrices based at $\id$. Then
  $$
  i_{RS} \bigl(\{P_1(t)\} \# \{P_2(t) \cdot P_1(1)\}\bigr) = i_{RS} \bigl(\{P_1(t)\} \bigr) + i_{RS} \bigl(\{P_2(t) \}\bigr) +D,
  $$
  where $D =- \frac{1}{2} \sign (M_{P_1(1)} + M_{P_2(1)})$.
\end{theorem}

Here the signature $\sign Q$ of a quadratic form $Q$ is taken to mean the number of negative squares minus the number of positive squares appearing in a diagonal form of $Q$.


\begin{proof}[Proof of Lemma \ref{lem-cat-1}]
  Proceed by induction on $l$. The case $l=1$ is trivial by (\ref{j=1}). For $l>1$ we have, using relation \eqref{inductive for j}, Theorem \ref{thm-PDG} and the inductive hypothesis,
  \begin{align*}
    i_{RS}(\{\overline{\Gamma}_l (t )\} )  &=
                                             i_{RS}\left(\{ \overline{\Gamma}_{l-1}(t) \} \# \{\Gamma_l(t) \} \right) \\
                                           &= i_{RS}\left(\{ \overline{\Gamma}_{l-1}(t) \} \# \{\Psi_l(t) \cdot \Gamma_{l-1}(1) \} \right) \\
                                           &= i_{RS}(\{ \overline{\Gamma}_{l-1}(t) \}) + i_{RS}\left(\{\Psi_l(t)\}\right)+D \\
                                           &= \sum_{j=1}^{l-1} i_{RS}(\{\Psi_j(t)\}) + i_{RS}(\{\Psi_l(t)\}) + D \\
                                           &= \sum_{j=1}^{l} i_{RS}(\{\Psi_j(t)\}) +D,
  \end{align*}
  %
  where $D =- \frac{1}{2} \sign(M_{\Psi_l(1)} + M_{\Gamma_{l-1}(1)})$.
  From here, we see that our proof is finished once we show that $D = 0$. Recalling that
  \begin{align*}
    \Psi_{l+1}(1) &= B(\beta_{l+1}) A(\alpha_{l+1}), \\
    \Gamma_l(1) &= B(\beta_{l}) A(\alpha_{l}) \cdot \ldots \cdot B(\beta_1) A(\alpha_1),
  \end{align*}
  we claim that both $M_{\Psi_{l+1}(1)}$ and $M_{\Gamma_l(1)}$ have the following asymptotic form as $k \to \infty$:
  \[\begin{pmatrix}
      0 & 1/2 \\ 1/2 & 0
    \end{pmatrix} + O\left(\frac{1}{k}\right).
  \]
  We only prove this formula for $M_{\Gamma_l(1)}$, as the proof for $M_{\Psi_{l+1}}(1)$ is similar and easier. First, in order to trace the exponential growth of $k$, for each $j \in \bN$, we will introduce $a_j$ and $b_j$ such that
  \begin{equation} \label{into-a}
    2L k N_j \eps_V(p) = \alpha_j = k a_j, \,\,\,\,\,\mbox{so} \,\,\,\,\,\, a_j = 2LN_j \eps_V(p)
  \end{equation}
  and
  \begin{equation} \label{into-b}
    2L k M_j \eps_H(p) = \beta_j = k b_j, \,\,\,\,\,\mbox{so} \,\,\,\,\,\, b_j = 2L M_j \eps_H(p)
  \end{equation}
  By explicit computation, we have the following asymptotic expansion for $\Gamma_{l}(1)$, where we set $c_l := (-1)^lb_l a_l ... b_1 a_1$ for brevity:
  \begin{equation}\label{eq-oB-asym}
    \Gamma_l(1) = \begin{pmatrix}
      k^{2l}c_l & -k^{2l-1} \frac{c_l}{a_1} \\ k^{2l-1} \frac{c_l}{b_l} & -k^{2l-2}\frac{c_l} {b_l a_1}
    \end{pmatrix} \cdot \left( 1 + O\left(\frac{1}{k}\right)\right), \quad k \to \infty.
  \end{equation}
  Indeed, this is simply Lemma \ref{lem-Abarj-asym} (for the case $i=1$, $j=r$) in different notation. To be precise, in \eqref{eq-oB-asym} we mean that each entry of $\Gamma_l(1)$ is multiplied by a different term of the form $1+O(1/k)$. Using this asymptotic formula (and the fact that $\det \Gamma_l(1) = 1$, as it is a product of shear matrices), one readily computes that
  $$
  M_{\Gamma_l(1)} = \begin{pmatrix}
    0 & 1/2 \\ 1/2 & 0
  \end{pmatrix} + O(1/k).
  $$
  As noted above, an identical argument shows that $M_{\Psi_{l+1}(1)}$ has this same asymptotic form, giving us
  $$
  Q := M_{\Gamma_l(1)} + M_{\Psi_{l+1}(1)}= \begin{pmatrix}
    0 & 1 \\ 1 & 0
  \end{pmatrix} + O(1/k).
  $$
  In particular, for $k \gg 1$ we get $\det Q < 0$, and so its eigenvalues have opposite signs. We deduce that $\sign Q = 0$, and our proof is complete.
\end{proof}

\subsubsection{Proof of Lemma \ref{lem-cat-2}}
The proof of Lemma \ref{lem-cat-2} is an application of another theorem regarding concatenation, which applies this time to degenerate paths.  Before stating this theorem, however, we must introduce yet more notation. (cf. Section 2.4 in \cite{DeGosson2008b}). For a fixed element $\psi \in Sp(2)$, define
\[ Sp_{\psi}(2) = \{ \Phi \in Sp(2) \,| \, \Phi - \psi\,\,\,\mbox{is invertible} \}. \]

Now fixing an element $\psi \in Sp_{\id}(2) =: Sp_0(2)$, we can define an operator on $Sp_{\psi}(2)$ by

\begin{equation}\label{ope}
  C_{\psi} (\Phi) = J (\psi - \id) (\Phi - \psi)^{-1} (\Phi - \id)
\end{equation}

For our purposes, we choose $\psi$ to be the matrix
\[  \psi = \left( \begin{array}{cc} 0 & 1\\ -1 & 0 \end{array} \right), \]
which is easily verified to be in $Sp_0(2)$.  For our $\psi$, $A(\alpha_j)$ is in $Sp_{\psi}(2)$ for any $j \in \bN$ when $k$ is sufficiently large. Indeed, we have
\[ \left( \begin{array}{cc} 1 & 0 \\ - k a_j & 1 \end{array} \right)  - \left( \begin{array}{cc} 0 & 1\\ -1 & 0 \end{array} \right)  = \left( \begin{array}{cc} 1 & -1\\ 1-k a_j & 1 \end{array} \right) \]
of determinant $2 - k a_j$, which is clearly non-zero for $k \gg 1$. Similarly, $B(\beta_j)$ is in $Sp_{\psi}(2)$. With these further notations established, we may now state the aforementioned theorem, which combines Corollary 3.7 and Lemma 3.3 from \cite{DeGosson2008b}.

\begin{theorem}\label{thm-PDG-2}
  Let $\{P_1(t)\}_{[0,1]}$ and $\{P_2(t)\}_{[0,1]}$ be two paths of symplectic matrices based at $\id$. If $P_1(1)$ and $P_2(1)$ are in $Sp_{\psi}(2)$, then
  \[ i_{RS}(\{P_1(t))\} \# \{P_2(t) \cdot P_1(1)\}) = i_{RS}(\{P_1(t)\}) + i_{RS}(\{P_2(t)\}) + Z \]
  where
  \[ Z = \frac{1}{2} \sign(C_{\psi}(P_2(1)) - C_{\psi}(P_1(1))) - \frac{1}{2} \sign(C_{\psi}(P_2(1))) + \frac{1}{2} \sign(C_{\psi}(P_1(1))). \]
\end{theorem}

We recall that the signature $\sign M$ of $M$ is the number of negative squares minus the number of positive squares.

\begin{proof}[Proof of Lemma \ref{lem-cat-2}]
  Start by setting
  \[ P_1(t) = A_t(\alpha_j) \,\,\,\,\,\mbox{and} \,\,\,\,\,\, P_2(t) = B_t(\beta_j) \]
  Then by Theorem \ref{thm-PDG-2}, Lemma \ref{lem-cat-2} follows if we can show that our $Z$ is zero. By definition (\ref{ope}), we have (for sufficiently large $k$)
  \begin{align*}
    C_{\psi}(A(\alpha_j)) &= \left( \begin{array}{cc} 0 & 1\\ -1 & 0 \end{array} \right) \left( \begin{array}{cc} -1 & 1\\ -1 & -1 \end{array} \right)\left( \begin{array}{cc} 1 & -1\\ 1-k a_j & 1 \end{array} \right)^{-1}  \left( \begin{array}{cc} 0 & 0\\ -k a_j  & 0 \end{array} \right)  \\
                          & = \left( \begin{array}{cc}  \frac{2 k a_j}{2 - k a_j} & 0\\ 0 & 0 \end{array} \right)
  \end{align*}

  and

  \begin{align*}
    C_{\psi}(B(\beta_j)) &= \left( \begin{array}{cc} 0 & 1\\ -1 & 0 \end{array} \right) \left( \begin{array}{cc} -1 & 1\\ -1 & -1 \end{array} \right)\left( \begin{array}{cc} 1 & k b_j -1\\ 1 & 1 \end{array} \right)^{-1}  \left( \begin{array}{cc} 0 & k b_j \\ 0 & 0 \end{array} \right)  \\
                         & = \left( \begin{array}{cc} 0 & 0\\ 0 & \frac{2 k b_j}{2 - k b_j} \end{array} \right)
  \end{align*}

  giving

  \[  C_{\psi}(B(\beta_j)) -  C_{\psi}(A(\alpha_j)) = \left( \begin{array}{cc}  -\frac{2 k a_j}{2 - k a_j} & 0\\ 0 &  \frac{2 k b_j}{2 - k b_j}\end{array} \right). \]

  From here, it is easy to compute, for large enough $k$,
  \begin{itemize}
  \item{} $\frac{1}{2} {\sign}(C_{\psi}(B(\beta_j)) -  C_{\psi}(A(\alpha_j))) = 0$;
  \item{} $\frac{1}{2} {\sign}(C_{\psi}(B(\beta_j))) =  \frac{1}{2}$;
  \item{} $\frac{1}{2} {\sign}( C_{\psi}(A(\alpha_j))) =  \frac{1}{2}$,
  \end{itemize}
  giving $Z = 0$.  This completes the proof of Lemma \ref{lem-cat-2}, and hence the proof of Lemma \ref{thm-cat-tot}.
\end{proof}

\subsubsection{Completing the computation}

With Lemma \ref{thm-cat-tot} in hand, we have only to compute the Robbin-Salamon indices of the paths $\{A_t(\alpha_j)\}$, $\{B_t(\beta_j)\}$. For this we quote \cite[4.1.1]{Gutt2013}, which computes for $\alpha_j \neq 0$,
\begin{equation}
  i_{RS} \bigl( \{A_t(\alpha_j)\} \bigr) = -\frac{1}{2} \,\sign\alpha_j,
\end{equation}
which, in conjunction with the conjugation invariance of $i_{RS}$, gives also
\begin{equation}\label{eq-CZ-shear}
  i_{RS} \bigl( \{B_t(\beta_j) \} \bigr) = -\frac{1}{2} \,\sign\beta_j.
\end{equation}

Combining this calculation with Theorem \ref{thm-cat-tot} and the definitions of $\alpha_j$ and $\beta_j$ we obtain
\begin{equation}
  \begin{split}
    i_{RS}(\{\overline{\Gamma}_r(t)\}) & = \sum_{j=1}^r \bigl[ i_{RS} \bigl( \{ A_t(\alpha_j) \} \bigr) + i_{RS} \bigl( \{B_t (\beta_j) \} \bigr) \bigr] \\
    & = -\frac{1}{2} \sum_{j=1}^{r} \bigl[ \sign (N_j \eps_V(z_{2j-2})) + \sign (M_j \eps_H(z_{2j-1}))  \bigr].
  \end{split}
\end{equation}

Noting that $\eps_V(z_{2j})=\eps_{2j}$ and $\eps_H(z_{2j-1})=\eps_{2j-1}$, we get finally:
\begin{equation}
  \mu_{CZ}(z(\vec{\eps})) =1 -i_{RS}(\{\overline{\Gamma}_r(t)\} =1 + \frac{1}{2} \sum_{j=1}^{r} \bigl[\eps_{2j-2} \sign (N_j)  + \eps_{2j-1}  \sign (M_j) \bigr].
\end{equation}
This completes the proof of Theorem \ref{thm-cz-sgn}.

\subsection{Floer data of the smooth egg-beater map}

Up until this point, we have computed the Floer data for the non-smooth homeomorphism $\phi=\phi_{k,w}$.
In this section we compute the Floer data for the actual smoothed Hamiltonian diffeomorphism $\Phi_k(w)$.
For this section only, we denote by $\Phi_{k, \delta} \colon \bF_2\langle V,H \rangle \to \Ham(\Sigma,\omega)$ the homomorphism sending $V$ to $f_{V,\delta}^k$ and $H$ to $f_{H,\delta}^k$.
Thus the egg beater homomorphism $\Phi_k$ defined in section \ref{sec-eb-morphism} is simply $\Phi_{k,1/k}$, and the non-smooth homeomorphism $\phi$ is just $\Phi_{k,0}(w)$.
For our fixed balanced word $w \in \bF_2\langle V ,H \rangle$, denote by $H_\delta$ the Hamiltonian generating $\Phi_{k,\delta}(w)$ obtained by concatenating the $\delta$-smoothed autonomous Hamiltonians.

\begin{proposition}\label{prop-smoothing}
There exist $\delta_0 = \delta_0(w)$ and $k_0 = k_0(w)$ such that for $0 < \delta < \delta_0$ and $k > k_0$ the fixed points of the smooth Hamiltonian diffeomorphism $\Phi_{k,\delta}(w)$ in the class $\alpha(k,w)$ coincide precisely with those of the homeomorphism $\Phi_{k,0}(w)$.
\end{proposition}

To begin with the proof of Proposition \ref{prop-smoothing}, we show that the asymptotic form we've proven for the fixed points can be easily shown to hold for any fixed points of a small enough smoothing.
\begin{lemma}\label{lem-smooth-asymp}
  For small enough $\delta > 0$, for any fixed point $z$ of $\Phi_{k,\delta}(w)$ in the class $\alpha(k,w)$ its intermediate points satisfy the asymptotics \eqref{eq-inter-pts-asym}. Moreover, for large enough $k\in \bN$, for every $t$ we have $H_\delta(t, \cdot) = H_0(t, \cdot)$ in a neighbourhood of $\phi^t(z)$.
\end{lemma}

This lemma follows immediately from the following estimate for the shear map $f_\delta \colon C_* \to C_*$. Denote $U = (-1,1) \times (-1,1)/L\Z \subset C_*$.
\begin{lemma}
  Suppose $(x,y) \in U$ is a point such that, for some $N \in \bZ, k \in \mathbb{N}$ and for $n =k \cdot \sign(N)$ one has $f_\delta^{kN}(x,y)\in U$ and $[\{f^t_\delta(x,y)\}_{t\in [0,kN]}] = n$ in $\pi_1(C_*,U) = \mathbb{Z}$. Then one has $u_\delta(x) = \nu + O(1/k)$, where $\nu = \frac{n}{2kN}=\frac{1}{2|N|}$.
\end{lemma}

\begin{proof}
  We have
  $$
  f_\delta^{kN}(x,y) = (x, [y + 2kNLu_\delta (x)]).
  $$
  On the other hand, considering the condition on $[\{f^t_\delta(x,y)\}_{t\in[0,kN]}]_{\pi_1(C_*,U)}$ and lifting to the universal cover, we get
  $$
  y+2kNLu_\delta(x) \in (nL-1,nL+1),
  $$
  so
  $$
  u_\delta(x) = \frac{n}{2kN} + O(1/k) = \nu + O(1/k)
  $$
\end{proof}
Therefore, for $\delta>0$ small enough (depending on $N$), we get that $2\delta < u_\delta(x) < 1-2\delta$. Since $||u_\delta - u_0|| < \delta$, this means that $x$ is away from the subset $U_\delta$ where $u_\delta \neq u_0$. In particular, $u_\delta(x)=u_0(x)=1-|x|$, and hence the bound becomes $|x|=1-\nu + O(1/k)$. Now, since we are using the concatenation Hamiltonian $H_\delta$ and our class $\talpha(k,w)$ is determined by \eqref{eq-mn} and \eqref{eq-munu}, the first claim in Lemma \ref{lem-smooth-asymp} follows. If follows that, for $k\in \bN$ large enough (depending on $w$), each intermediate points (and hence its intermediate path) remains in the region of the annulus where $H_\delta$ coincides with the appropriate shear Hamiltonian. This proves the second claim in Lemma \ref{lem-smooth-asymp}. In particular, it follows that the computations of actions and indices we carried out for the non-smooth isotopy are valid also for the smoothed isotopy, for large enough $k$ and small enough $\delta > 0$. We summarize this in the following proposition.

\begin{proposition}\label{prop-Floer-smooth}
  Let $w = H^{M_r} V^{N_r} \cdots H^{M_1} V^{N_1} $ be a balanced word. There exists $k_0 = k_0(w)$ such that for $k>k_0(w)$, $\Phi_k(w)$ has exactly $2^{2r}$ fixed points in the class $\alpha(k,w)$, indexed by $z(\vec{\eps})$, $ \vec{\eps}\in\{\pm 1\}^{2r} $. The action and Conley-Zehnder index of the fixed point $z(\vec{\eps})$ are given by \eqref{eq-action} and \eqref{eq-CZ}, respectively.
\end{proposition}

\subsection{Linear growth of the boundary depth}\label{subsect-linear-growth}
 We are ready to prove Theorem \ref{norm}. We first prove it for a surface, and then, as an immediate corollary, for the general case.

{\bf Proof for $(\Sigma,\omega)$}. We note that $\Sigma$ is $\alpha$-atoroidal for any $\alpha \in \pi_0(\cL \Sigma)$, since any map $T^2 \to \Sigma$ is has degree zero.

First we take a word $ w \in \bF_2$ which is conjugate to $H^n$ or $V^n$ for some $n \in \bZ$. By conjugation invariance of the Hofer norm, we can assume that $w=H^n$ or $w=V^n$ for some $n \in \bZ$. Assume for concreteness that $w=V^n$. Then $\phi:= \Phi_{k}(w)$ is the time-one map of the autonomous Hamiltonian $k H_V$, where $H_V$ is the normalized (smoothed) Hamiltonian corresponding to the $n$-th power of the shear map around $C_V$. Observe that the Hamiltonian flow generated by $kH_V$ has no contractible periodic orbits in $C_V$, and $kH_V$ is locally constant outside $C_V$, where it attains its extremal values. Therefore,
$$
\spec(\phi) = \{\min k H_V, \max kH_V\}.
$$
We focus on the spectral invariants $c_+(\phi)$ and $c_-(\phi)$. As $c_\pm (\phi) \in \spec(\phi)$, and by \eqref{eq-nu-positive} $c_+(\phi) \neq c_-(\phi)$ we deduce that at least one of $c_\pm(\phi)$ is non-zero, and hence is linear in $k$. Since $|c_\pm(\phi)|\leq ||\phi||_\rH$, we obtain a lower bound for $||\Phi_k(w)||_\rH$ linear in $k$. This concludes the proof of Theorem \ref{norm} for words conjugate to a power of $V$ or a power of $H$.
Next take a word $w \in \bF_2$ which is not conjugate to $H^n$ or $V^n$. Then $w$ is conjugate to a balanced word $H^{M_r} V^{N_r} \cdots H^{M_1} V^{N_1}$, where $N_j, M_j \neq 0$ and $r \geq 1$.  Let $\alpha={\alpha}(w, k)$ be our specially chosen compatible class. We consider the Floer complex $CF_*(H)_{\alpha}$ of the Hamiltonian $H$ generating $\Phi_k(w)$ obtained by concatenating the (autonomous) Hamiltonians used to define $f_{V,1/k}$ and $f_{H,1/k}$. Our conventions for the choice of representative $\eta_{\alpha}$ and trivialization of the bundle $\eta_{\alpha}^*TM$ are as in Section \ref{action-cz}. The next Lemma is the key step in the proof of Theorem \ref{norm}.
\begin{lemma}\label{keylemma} For sufficiently large $k \in \bN$ we have
\begin{equation}\label{eq-beta-bound}
\beta_{\alpha}(\Phi_k(w) ) \geq  (L k / 4)  \min \{ |N_j| , |M_j| \}.
\end{equation}
\end{lemma}

Here $\beta_{\alpha}$ denotes the boundary depth, as discussed in Section \ref{Floer}. Note that for a balanced word $w$ the quantity $\min\{|N_j|, |M_j|\}$ (which was denoted $\tau(w)$ in \eqref{eq-tau}) is strictly positive, hence the right hand side of \eqref{eq-beta-bound} is linear in $k$.
\begin{proof}
Using Proposition \ref{prop-Floer-smooth}, we know that the fixed points in the class $\alpha(k,w)$ are indexed by $\vec{\eps}\in \{\pm 1\}^{2r}$, with actions and Conley-Zehnder indices given by \eqref{eq-action} and \eqref{eq-CZ}, which we recall:
\begin{equation*}
 \cA(z(\vec{\eps})) = L k\sum_{j=1}^{r} \left[ \eps_{2j-2} N_j \left(1-\nu_j\right)^2 + \eps_{2j-1} M_j \left(1-\mu_j \right)^2\right] + O(1),
\end{equation*}
and
\begin{equation*}
\mu_{CZ}(z(\vec{\eps})) = 1 + \frac{1}{2} \sum_{j=1}^{r} \bigl[\eps_{2j-2} \sign (N_j)  + \eps_{2j-1}  \sign (M_j) \bigr].
 \end{equation*}
Let $z_0$ be the unique fixed points of maximal Conley-Zehnder index $\mu_{CZ}(z_0)=1+r$. Note first that $z_0$ is not a Floer boundary (as it has maximal index), and hence, since Floer homology in the non-trivial class $\alpha$ vanishes, $\partial z_0 \neq 0$. So $\del z_0$ is a non-trivial linear combination of critical points of index $r$. Therefore,
$$
\cA(z_0) - \cA(\partial z_0) \geq \Delta := \cA(z_0) - \max\{\ \cA(z) : \mu_{CZ}(z)=r\}.
$$
But if $z$ is a fixed point of Conley-Zehnder index $\mu_{CZ}(z)=r$, then the sign vector for $z$ only differs from the sign vector for $z_0$ in one entry. Hence we have the following asymptotics for the action difference:
\[ \cA(z_0)-\cA(z)  \geq  \frac{L k}{2} \text{min} \{ |N_j| , |M_j| \} + O(1). \]
So for large enough $k$, we have
$$
\Delta > \frac{Lk}{4} \text{min} \{ |N_j| , |M_j| \}.
$$
Noting that $\partial^{-1}(\partial z_0) = \{z_0\}$, we see that
$$
\inf_{y \in \partial^{-1}(\partial z_0)} \left( \cA(y) - \cA(\partial z_0)\right) \geq \Delta.
$$
Therefore, by definition \eqref{eq-bd-def}, for large enough $k$ we get the desired bound
$$
\beta_\alpha (\Phi_w(z)) \geq \Delta \geq \frac{Lk}{4} \text{min} \{ |N_j| , |M_j| \}.
$$

\end{proof}
By applying \eqref{bdydepth}, we obtain the same lower bound on the Hofer norm of $\Phi_{k}(w)$. This concludes the proof of Theorem \ref{norm} for all words $w \in \bF_2$. Note also that for long words $w$, the lower bound given by the right hand side \eqref{eq-beta-bound} is simply $(L/4) \tau(w) \cdot k $, thus proving Remark \ref{tau}.\\ \\

{\bf Proof for the case $\Sigma \times M$.} We first note that, denoting $\alpha_0 \in \pi_0(\cL M)$ the class of the constant loop, the manifold $\Sigma \times M$ is $\alpha \times \alpha_0$-atoroidal for any $\alpha \in \pi_0(\cL\Sigma)$. The computations of of the fixed points and their actions and indices carry over to this case without change, except that the Conley-Zehnder indices of all fixed points are shifted by $\frac{1}{2}\dim M$, which doesn't affect the arguments.
We again consider separately two cases. First, when $w$ is conjugate to $V^n$ or $H^n$, for some $n\in\Z$. Then the argument outlined above using spectral invariants applies verbatim to show that $||\phi_{k,w}\times \id_M||_{\mathrm{H}}$ grows linearly in $k$.
Next, for a long word $w$, we assume without loss of generality that $w$ is balanced. Then using \eqref{eq-bd-product} we have $\beta(\phi_{k,w} \times \id_M) \geq \beta(\phi_{k,w})$ and thus grows linearly in $k$, whence so does $||\phi_{k,w}\times \id_M||_{\mathrm{H}}$.

\printbibliography

\end{document}